\documentclass[oneside,a4paper]{amsart}
\usepackage{amsthm, amsmath, amssymb, amscd,  amsfonts, amstext,
  amsbsy}
\usepackage[all]{xy}
\usepackage{mathrsfs}
\usepackage{enumerate}
\usepackage{rotate}
\usepackage{graphicx}

\makeindex

\newcommand{\pow}{\ensuremath{\mathscr{P}}}
\newcommand{\ZZ}{\ensuremath{\mathbb{Z}}}
\newcommand{\RR}{\ensuremath{\mathbb{R}}}

\def\leng{{\rm lh}}

\def\Cl{{\rm Cl}}

\def\id{{\rm id}}
\newcommand{\Bai}{\ensuremath{{}^\omega \omega}}
\newcommand{\Can}{\ensuremath{{}^\omega 2}}
\newcommand{\seqo}{\ensuremath{{}^{<\omega}\omega}}
\newcommand{\ZFC}{\ensuremath{{\rm \mathsf{ZFC}}}}
\newcommand{\AD}{\ensuremath{{\rm \mathsf{AD}}}}
\newcommand{\ZF}{\ensuremath{{\rm \mathsf{ZF}}}}

\newcommand{\DCR}{\ensuremath{{\rm \mathsf{DC}(\mathbb{R})}}}
\newcommand{\AC}{\ensuremath{{\rm \mathsf{AC}}}}
\newcommand{\ACOR}{\ensuremath{{\rm \mathsf{AC}_\omega(\mathbb{R})}}}

\newcommand{\SLOL}{\ensuremath{{\rm \mathsf{SLO^L}}}}
\newcommand{\SLOW}{\ensuremath{{\rm \mathsf{SLO^W}}}}
\newcommand{\SLO}{\ensuremath{{\rm \mathsf{SLO}}}}

\newcommand{\NFS}{\ensuremath{{\rm \lnot \mathsf{FS}}}}

\newcommand{\pII}{\ensuremath{{\rm \mathbf{II}}}}
\newcommand{\conc}{{}^\smallfrown}

\newcommand{\rank}[1]{\| #1 \|}

\newcommand{\imp}{\Rightarrow}

\newcommand{\fhi}{\varphi}
\newcommand{\bSigma}{\mathbf{\Sigma}}
\newcommand{\bPi}{\mathbf{\Pi}}
\newcommand{\bGamma}{\ensuremath{\mathbf{\Gamma}}}
\newcommand{\bDelta}{\mathbf{\Delta}}
\newcommand{\bLambda}{\mathbf{\Lambda}}
\newcommand{\bN}{\mathbf{N}}

\newcommand{\F}{\mathcal{F}}
\newcommand{\G}{\mathcal{G}}

\newcommand{\B}{\mathcal{B}}

\newcommand{\Lip}{\ensuremath{{\mathsf{Lip}}}}
\newcommand{\Bor}{\ensuremath{{\mathsf{Bor}}}}
\renewcommand{\L}{\ensuremath{{\mathsf{L}}}}
\newcommand{\W}{\ensuremath{{\mathsf{W}}}}
\newcommand{\Sat}{{\rm Sat}}
\newcommand{\restr}[2]{#1 \restriction #2}
\renewcommand{\sf}[1]{\ensuremath{{\mathsf{#1}}}}
\newcommand{\seq}[2]{\langle #1 \mid  #2 \rangle}
\newcommand{\onto}{\twoheadrightarrow}

\newtheorem{theorem}{Theorem}[section]
\newtheorem{lemma}[theorem]{Lemma}

\newtheorem{corollary}[theorem]{Corollary}
\newtheorem{proposition}[theorem]{Proposition}

\theoremstyle{definition}
\newtheorem{question}{Question}
\newtheorem{claim}{Claim}[theorem]
\newtheorem{defin}{Definition}

\newtheorem{conj}{Conjecture}

\theoremstyle{remark}
\newtheorem{remark}[theorem]{Remark}

\begin{document}

\title{Borel-amenable reducibilities for sets of reals}
\author{Luca Motto Ros}
\date{\today}
\address{Kurt G\"odel Research Center fo Mathematical Logic\\ University of Vienna \\ Austria}
\email{luca.mottoros@libero.it}
\thanks{Research partially supported by FWF Grant P 19898-N18. The author would like to thank his PhD\ thesis advisors Alessandro
Andretta and Riccardo
Camerlo for all the useful suggestions and  stimulating discussions on Wadge
theory.}

\begin{abstract}
  We show that if $\F$ is any ``well-behaved''
  subset of the Borel functions and we assume the Axiom of Determinacy then the
  hierarchy of degrees on 
  $\pow(\Bai)$ induced by $\F$ turns out to look like
  the Wadge hierarchy (which is the special case where $\F$
is the set of continuous functions).
\end{abstract}

\maketitle

\section{Introduction}

Intuitively, a set $A$ is \emph{simpler} than --- or as complex as
--- a set $B$ if the problem of verifying membership in $A$
can be reduced to the problem of verifying membership in $B$. In
particular, if $X$ is a Polish space and $A,B \subseteq X$, we say that $A$ is \emph{(continuously) reducible}
to $B$ just in case there is a continuous function $f$ such that
$x \in A$ if and only if $f(x) \in B$ for every $x \in X$, in
symbols $A \leq_\W B$. (The symbol ``$\W$'' is for W.~Wadge
who started a systematic study of this relation in his
\cite{wadgethesis}.)
Thus in this setup continuous functions are used as
\emph{reductions} between subsets of $X$. 
The equivalence classes of the equivalence relation
induced by $\leq_\W$ on the Baire space $\Bai$ are called
\emph{Wadge degrees}, and the preorder $\leq_\W$ induces a partial order
$\leq$ on them. Using game theoretic techniques, Wadge proved a simple but fundamental  Lemma which has played a key role in
various parts of Descriptive Set Theory: $\AD$, the \emph{Axiom of 
Determinacy}, implies that if $A,B \subseteq \Bai$ then
\begin{equation}\label{eqSLOW}
\tag{$\star$} A \leq_\W B \quad \vee \quad \Bai
    \setminus B \leq_\W A.
\end{equation}
Wadge's Lemma says that $\leq_\W$ is a semi-linear order, therefore 
\eqref{eqSLOW} is usually denoted by $\SLOW$.
Starting from this result, Wadge (and many other set theorists
after him) extensively studied the
preorder $\leq_\W$ and gave  under $\AD+\DCR$ a complete description of  
the structure of
the Wadge degrees (and also of  the \emph{Lipschitz degrees} --- 
see Section \ref{sectionpreliminaries}).
In \cite{andrettamartin} and \cite{andrettamoreonwadge}   A.~Andretta and 
D.~A.~Martin 
considered  Borel reductions and $\bDelta^0_2$-reductions instead of 
continuous
reductions: using topological arguments (mixed with
game-theoretic techniques in the second case), 
they showed that the degree-structures induced by these reducibility notions look exactly like
 the Wadge hierarchy.
Thus a natural question arises:

\begin{question}
  Given any reasonable set of functions $\F$ from the Baire space into 
  itself, which kind of structure
  of degrees is induced if the functions from $\F$ are used as
  reductions between sets (i.e.\ if we consider the preorder $A
  \leq_\F B \iff {A = 
  f^{-1}(B)}$ for some $f \in \F$)?
\end{question}

The term ``reasonable'' is a bit vague, but
should be at least such that all  ``natural'' sets of functions, such as
continuous functions, Borel functions and so on,
 are reasonable (these sets of functions will be called here \emph{Borel-amenable} --- see Section
\ref{sectionborelamenability} for the definition).

In this paper we will answer to the previous Question 
for the Borel context, 
i.e.\ when $\F
\subseteq \Bor$: we will prove that under (a weakening of) $\AD+\DCR$ each of these sets of 
functions yields a semi-linear ordered 
stratification of degrees, and provides in this way a corresponding 
notion of complexity on $\pow(\Bai)$.  In particular, we will prove in Sections \ref{sectionborelamenability} and \ref{sectionDP} that all
these  degree-structures turn out
to look like the Wadge one (which 
can be determined as a particular instance
of our results).
The new key idea used in this paper is
 the notion of 
\emph{characteristic set} $\Delta_\F$ of the collection $\F$, which basically contains all the
subsets of $\Bai$ which are simple from the ``point of view'' of $\F$ (see Section \ref{sectionsetsofreductions} for the precise definition).
This tool is in some sense crucial for the study of the $\F$-hierarchy:
in fact, if one knows the characteristic set 
of a Borel-amenable set of functions (even without any other information about 
the set $\F$), then one can completely describe the hierarchy of degrees
induced by $\F$. As an application it turns out that for distinct $\F$ and $\G$, their degree-hierarchies coincide just in case 
$\Delta_\F = \Delta_\G$ (see Section 
\ref{sectionborelamenability} again). 
Moreover in Section \ref{sectionBAR} we will analyse the collection of all the Borel-amenable sets of functions and provide several examples (towards
this goal we will also  answer negatively to a question 
about generalizations of the
Jayne-Rogers Theorem posed by Andretta in his
\cite{andrettaslo}). Finally, in Section 
\ref{sectionsuccessor} we will show how to define
some operations which allow to construct, given a certain
degree in the $\F$-hierarchy, 
its successor degree(s): this will
give a more combinatorial description of the degree-structure 
induced by $\F$. 
In a future paper we will show, building on the results obtained in this paper, how to extend our analysis of Borel reductions to a wider class of sets of functions.

The present work is, in a sense, the natural extension of \cite{andrettamartin}, and we assume the reader is familiar with the arguments contained therein.


\section{Preliminaries}\label{sectionpreliminaries}

For the sake of precision our base theory will be  always
$\ZF+\ACOR$, 
and  we will specify  which
auxiliary axioms
are used for each statement. Nevertheless one should keep in mind that
\emph{all the  results
of this paper are true under $\AD+\DCR$}, or even just under $\SLOW+\NFS+\DCR$, where $\NFS$ is the statement\footnote{The 
axiom $\NFS$ is a consequence of both the statements ``every set of reals 
has the Baire property''
 and ``every set of reals is Lebesgue measurable''  (hence it is 
also a consequence of $\AD$), but it is weaker than them. Moreover it is 
consistent both with the Perfect Subset Property and its negation.}
``there are no flip-sets'' (recall that a subset $F$ of the Cantor space
$\Can$ is a \emph{flip-set} just in case for every $z,w \in \Can$ such 
that $\exists ! n (z(n) \neq w(n))$ one has $z \in F \iff w \notin F$). 
In the latter case, recall also from \cite{andrettaequivalence} and
\cite{andrettamoreonwadge} that $\SLOW+\NFS+\DCR$ is equivalent to $\SLOL
+\NFS+\DCR$.
Finally, one should also observe that all the 
``determinacy axioms'' are
used in a
\emph{local way} throughout the paper: this means that our results hold for the sets in some
  (suitable) pointclass $\Gamma \subseteq \pow(\Bai)$ whenever we assume that
  the corresponding axioms hold for sets in $\Gamma$. In particular, if
we are content to shrink our hierarchies to the Borel subsets 
of $\Bai$, then we only need
Borel-determinacy (i.e.\ the determinateness of those games 
whose pay-off set is Borel). 
Therefore our results hold also in $\ZFC$ if we completely restrict our attention to the Borel 
context, that is if we compare only Borel sets with Borel functions.

\subsection{Notation}

Our notation is quite standard --- for all undefined symbols and
notions we refer the reader to the standard monograph \cite{kechris}
and to the survey paper \cite{andrettaslo}. The set of the natural numbers
will be denoted by $\omega$. Given a pair of set $A,B$, we will denote by
${}^B A$  
the set of all the functions from $B$ to $A$. In particular, ${}^\omega A$ will denote the collection of all the $\omega$-sequences of elements of $A$, while the
collection of the \emph{finite} sequences of natural numbers will be denoted
by $\seqo$ (we will refer to the length of a sequence $s$  
 with the symbol $\leng(s)$). 
The space $\Bai$ (the collection of the 
$\omega$-sequences of natural numbers) is called \emph{Baire space}, and as customary
in this subject we will always identify $\RR$ with it,
that is we put $\RR =\Bai$.
 The Baire space
is endowed with the topology induced by 
 the metric defined by $d(x,y) = 0$ if $x=y$ and $d(x,y) = 2^{-n}$, where $n$ is least such that $x(n) \neq y(n)$, otherwise. 
In particular, the basic open neighborhood of $\RR$ are of the form
$\bN_s = \{x \in \RR \mid s\subseteq x\}$
(for some $s \in \seqo$). Given $A \subseteq \RR$ we put $\neg A = \RR \setminus A$, and if $s \in \seqo$ we
put
$s\conc A = \{s\conc x \mid x \in A\}$
 (when $s = \langle n \rangle$ we will simply write $n \conc A$).
 Given $A_n,A,B \subseteq \RR$ we define $\bigoplus\nolimits_n A_n = \bigcup\nolimits_n (n \conc A_n)$
and $A \oplus B = \bigoplus\nolimits_n C_n$,
where $C_{2k} = A$ and $C_{2k+1} = B$ for every $k \in \omega$. Moreover 
for any $n,k \in \omega$, we put
$\vec{n} = \langle n,n,n,\dotsc \rangle$
and $n^{(k)}= \langle \underbrace{n, \dotsc, n}_k \rangle$.

A \emph{pointclass} (for $\RR$) is simply a non-empty $\Gamma
\subseteq \pow(\RR)$, while a \emph{boldface
  pointclass} $\bGamma$ is a
pointclass  closed under continuous preimage. If $\bGamma$ is a boldface pointclass 
then so is its \emph{dual} $\breve{\bGamma} = \{ \neg A \mid A \in
\bGamma\}$.
A boldface pointclass is \emph{selfdual} if it
coincides with its dual, otherwise it is
nonselfdual.
Finally, recall that a  boldface pointclass $\bGamma$ is said to have
(or admits) a universal set if there is some $U
\subseteq \RR \times \RR$ which is universal for $\bGamma$ and such that the
image of $U$ under the standard homeomorphism $\RR \times \RR \simeq
\RR$ is in $\bGamma$.

Let $\Gamma$ be any pointclass and let $D \subseteq \RR$. A
\emph{$\Gamma$-partition} of $D$ is a family
 $\langle D_n \mid n \in \omega \rangle$ of pairwise disjoint sets
 of $\Gamma$ such that $D = \bigcup_{n\in \omega} D_n$, 
and it is said to be \emph{proper} if at least two of the $D_n$'s are nonempty.

Let 
$\F \subseteq {}^\RR \RR$ be any set of
functions. Let 
$f\colon\RR \to \RR$ be an
arbitrary function and $\mathcal{C} = \langle C_n \mid n \in \omega \rangle$
 be some partition of $\RR$. We say that $f$ is
\emph{(locally) in $\F$ on the partition
  $\mathcal{C}$}
if there is a family of functions $\{f_n \mid n \in \omega\} \subseteq \F$
such that $\restr{f}{C_n} = \restr{f_n}{C_n}$ for every $n$.
Moreover, if $\Gamma \subseteq \pow(\RR)$  is any pointclass,
 we will say that $f$ is
\emph{(locally) in $\F$ on a $\Gamma$-partition} if there is some
$\Gamma$-partition such that $f$ is locally in $\F$ on it.

Given a positive real number $C$, we denote by
 $\Lip(C)$ the collection of all functions $f\colon \RR \to
\RR$ which are Lipschitz with
constant less or equal than $C$ (observe that we can always assume
$C = 2^k$ for some $k \in \ZZ$), and put $\Lip =
\bigcup_{k \in \ZZ} \Lip(2^k)$. Moreover, since it plays a special role in
the
theory of reductions, we denote the set $\Lip(1)$ with the
special symbol $\L$. Finally, we denote by  $\W$ the set of the
continuous functions, 
by $\sf{D}_\xi$ (for some $1 \leq \xi < \omega_1$) the set of 
all the \emph{$\bDelta^0_\xi$-functions} (i.e.\ of those $f$
such that $f^{-1}(D) \in \bDelta^0_\xi$ for every $D \in \bDelta^0_\xi$
or, equivalently, such that 
$f^{-1}(S) \in \bSigma^0_\xi$ for every $S \in \bSigma^0_\xi$), 
and by $\Bor$ the set of all the Borel functions 
(but for simplicity of notation we will sometimes put
$\sf{D}_{\omega_1} = \Bor$).

\subsection{Reducibilities}

We recall here the terminology about reducibilities for sets of reals
as presented in \cite{andrettamartin} (with some minor modifications).
Given a family of functions $\F \subseteq
{}^\RR \RR$, we would like to
use the
functions from $\F$ as \emph{reductions} and say that for every pair
of sets $A,B
\subseteq \RR$, the set $A$ is \emph{$\F$-reducible to
  $B$}
($A \leq_\F B$,
in symbols) if and only if there is some function $f \in \F$ such that
$A = f^{-1}(B)$, i.e.\ such that $\forall x \in \RR(x \in A \iff f(x) \in
B)$. Notice that $A \leq_\F B \iff \neg A \leq_\F \neg B$. 
Clearly we can also introduce the strict relation corresponding to $\leq_\F$ by letting
$A <_\F B \iff {A \leq_\F B} \wedge {B \nleq_\F A}$.
 Since in order to have degrees we would like to have $\leq_\F$ be a preorder
(i.e.\
reflexive and transitive),
\emph{we will always assume without explicitly mentioning it that each set of functions considered is closed under composition and
contains the identity function $\id$}. Under this assumption, we
can consider the equivalence relation $\equiv_\F$ canonically induced by $\leq_\W$
and call \emph{$\F$-degree} any equivalence class of
$\equiv_\F$ ($ [A]_\F$
will denote the $\F$-degree of $A$). 
A set $A$ is \emph{$\F$-selfdual} if and only if
$A \leq_\F \neg A$ (if and
only if $A \equiv_\F \neg A$), otherwise it is
\emph{$\F$-nonselfdual}. Since selfduality is
invariant under $\equiv_\F$, the definition can be applied to $\F$-degrees as
well. The \emph{dual} of $[A]_\F$ is $[\neg A]_\F$, and a pair of
distinct degrees of the form $\{[A]_\F,[\neg A]_\F\}$ is a
\emph{nonselfdual pair}. The preorder
$\leq_\F$ canonically induces a partial order $\leq$ on the $\F$-degrees (the strict part of $\leq$ will be denoted by $<$). Notice also that if $\F \subseteq \G \subseteq {}^\RR \RR$, then the
preorder $\leq_\G$ is coarser than $\leq_\F$: hence $A \leq_\F B \imp
A \leq_\G B$, if $A$ is $\F$-selfdual then it is also $\G$-selfdual,
and $[A]_\F \subseteq [A]_\G$.

If $\F$ contains all the constant functions, then
$[\RR]_\F=\{ \RR\}$ and $[\emptyset]_\F= \{
\emptyset\}$ are the $<$-least $\F$-degrees and form a nonselfdual
pair. We say\footnote{For the sake of simplicity, all the terminology of
  this paragraph will be often applied in the obvious way to
  sets (rather than to $\F$-degrees).} that $[A]_\F$
is a \emph{successor
  degree} if there is a degree $[B]_\F < [A]_\F$
for which there is no $C \subseteq \RR$ such that $[B]_\F <
[C]_\F < [A]_\F$ (such an $[A]_\F$ will be called 
\emph{successor} of $[B]_\F$). If an $\F$-degree is not a
successor and it is neither $[\RR]_\F$ nor $[ \emptyset]_\F$ (where
$\F$ contains all the constant functions again),
then we say that it is a \emph{limit degree}. A
 degree $[A]_\F$ is of \emph{countable
  cofinality} if it is minimal in the collection
of the  upper
bounds of a family of degrees $\mathcal{A}=\{[A_n]_\F \mid n \in \omega
\}$ each of which is strictly smaller than $[A]_\F$, i.e.\ if $[A_n]_\F <
[A]_\F$ for every $n \in \omega$, and for
every $[B]_\F$ such that $[A_n]_\F \leq [B]_\F$ (for every $n \in
\omega$) we have that $[B]_\F \nless [A]_\F$ 
(observe that if $[A]_\F$ is limit the definition
  given here is equivalent to requiring that $[A]_\F$ is
minimal among the  upper bounds of a chain $[A_0]_\F < [A_1]_\F <
\dotsc$).
If this is not the case then
$[A]_\F$ (is limit and) is said to be of \emph{uncountable
  cofinality}. 

The \emph{Semi-Linear Ordering Principle} for $\F$ is
\begin{equation}
\tag{$\SLO^\F$}
\forall A,B \subseteq \RR
({A \leq_\F B} \vee
  {\neg B \leq_\F A}).
\end{equation}

This principle implies that
if $A$ is $\F$-selfdual then 
$[A]_\F$ is comparable with all the other $\F$-degrees, while if
$A$ and $B$ are $\leq_\F$-incomparable
then  $\{[A]_\F,[B]_\F\}$ is a nonselfdual pair, i.e.\ $[B]_\F = [\neg A]_\F$.
Thus the ordering induced on the $\F$-degrees is \emph{almost} a
linear-order: this is the reason for which the principle $\SLO^\F$ is called 
``Semi-Linear Ordering Principle'' for $\F$.

\begin{lemma} \label{lemmabasic}
Let $\F \subseteq \G$ be two sets of functions from $\RR$ to
$\RR$. Then
$\SLO^\F \imp \SLO^\G$ and if we assume $\SLO^\F$
\[ \forall A,B \subseteq \RR (A<_\G B \imp A <_\F B).\]
\end{lemma}

\begin{proof}
The first part is obvious, since $A \leq_\F B \imp A \leq_\G B$. For
the second one, notice that $B
\nleq_\F A$ (otherwise $B \leq_\G A$). Since $\SLO^\F
\imp\SLO^\G$, $A <_\G \neg B$: if $A \nleq_\F B$ then $\neg B
\leq_\F A$ by $\SLO^\F$, and hence $\neg B \leq_\G A$, a contradiction!
\end{proof}

Recall also that if $\F$ is not too large and
$\SLO^\F$ holds,
then there is a uniform way to construct from a set $A \subseteq \RR$
a new set $J_\F(A)$ which is $\F$-larger than $A$ and $\neg A$ ($J_\F$
is also called \emph{Solovay's jump operator}). 

\begin{lemma}[Solovay]\label{lemmajumpoperator}
  Suppose that there is a
  surjection $j \colon \RR \onto \F$ and that $\SLO^\F$ holds. Then there is
  a map $J = J_\F\colon\pow(\RR) \to \pow(\RR)$ such that
\[ \forall A \subseteq \RR ({A<_\F J(A)} \wedge {\neg A <_\F J(A)}).\]
\end{lemma}

\subsection{Wadge and Lipschitz degrees}

\emph{We will assume $\SLOL+\NFS+\DCR$ for the rest of this Section} and state (without proof) some basic facts which will be useful for the next Sections. For more details see \cite{wadgethesis}, \cite{vanwesepwadgedegrees} or \cite{andrettaslo}.
 
If $A_0, A_1, \dotsc \subseteq \RR$ 
are such that $\forall n \exists m>n (A_m \nleq_\L A_n)$,
then $\bigoplus_n A_n$ is $\L$-selfdual and 
$[\bigoplus_n A_n]_\L$ is the least upper bound of the
$[A_n]_\L$'s. In particular, 
$A \oplus \neg A$ is always $\L$-selfdual and 
$[A \oplus \neg A]_\L$ is the least degree above $[A]_\L$ and
$[\neg A]_\L$. Moreover, after a selfdual $\L$-degree there is always another selfdual $\L$-degree, and 
a limit
$\L$-degree is selfdual if and only if it is of countable
cofinality, otherwise it is nonselfdual. 
Finally, the $\L$-hierarchy is well-founded 
($\rank{A}_\L$ will denote the canonical rank of the set $A$ with respect to $\leq_\L$)
and its antichains have size at most $2$. Therefore the Lipschitz hierarchy looks like this:
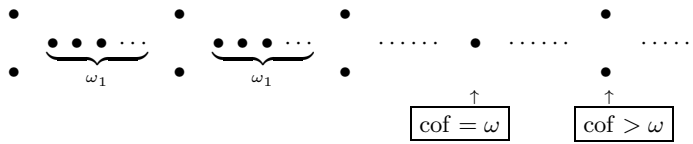
\begin{small}
\begin{equation*}\label{pictureLipschitz}\index{Hierarchy!Lipschitz}
\begin{array}{llllllllll}
\bullet & & \bullet & & \bullet & & & &\bullet
\\
& \smash[b]{\underbrace{\bullet \; \bullet \;
\bullet \; \cdots}_{\omega_1}} & &
\smash[b]{\underbrace{\bullet \; \bullet \; \bullet \;
\cdots}_{\omega_1}} & & \cdots\cdots &
\bullet & \cdots\cdots & & \cdots\cdots
\\
\bullet & & \bullet & & \bullet & & & & \bullet
\\
& & & & & &
\stackrel{\uparrow}{\makebox[0pt][r]{\framebox{${\rm cof} = \omega$}}} & &
\, \stackrel{\uparrow}{\makebox[0pt][l]{\framebox{${\rm cof} > \omega$}}}
\end{array}
\end{equation*}
\end{small}

The description of the Wadge hierarchy can be obtained from the
Lipschitz one using
the Steel-Van Wesep Theorem (see e.g.\ Theorem 3.1 in \cite{vanwesepwadgedegrees}).
A degree
$[A]_\W$ is nonselfdual if and only if $[A]_\L$ is nonselfdual
    (and in this case $[A]_\W=[A]_\L$), while every selfdual degree $[A]_\W$ 
is exactly the union of an  $\omega_1$-block of consecutive Lipschitz degrees: therefore
nonselfdual pairs and single selfdual degrees alternate in the $\W$-hierarchy. 
Moreover, the $\W$-hierarchy is well-founded and at
  limit levels of countable cofinality there is a single selfdual 
  degree, while at limit levels of uncountable cofinality there is a
  nonselfdual pair.
Hence the structure of the Wadge degrees looks like this:

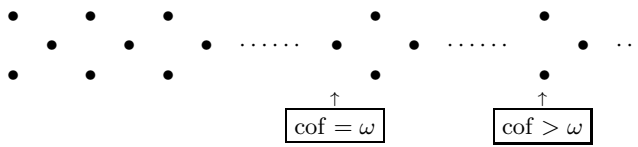
\begin{small}
\begin{equation*}\label{pictureWadge}\index{Hierarchy!Wadge}
\begin{array}{llllllllllllll}
\bullet & & \bullet & & \bullet & & & & \bullet & & & \bullet
\\
& \bullet & & \bullet & & \bullet & \cdots \cdots & \bullet
& & \bullet &\cdots\cdots & & \bullet & \cdots
\\
\bullet & & \bullet & & \bullet & & & & \bullet & & & \bullet
\\
& & & & & & &
\, \makebox[0pt]{$\stackrel{\uparrow}{\framebox{${\rm cof} = \omega$}}$}
& & & &
\, \makebox[0pt]{$\stackrel{\uparrow}{\framebox{${\rm cof} > \omega$}}$}
\end{array}
\end{equation*}
\end{small}

Finally, recall also that
the length of the Wadge hierarchy (as well as the length of the Lipschitz one)
is $\Theta = \sup\{\alpha \mid \text{there is a surjective }f\colon\RR \onto
\alpha\}$.

\section{Sets of reductions}
\label{sectionsetsofreductions}

The idea behind the next definition is that we would like to use the
functions from some $\F \subseteq{}^\RR \RR$ as \emph{reductions}, and
 study the relation
$\leq_\F$. In order to have a
nontrivial structure, we must require $\F$ to be neither too small nor
too large. For instance, we can not let $\F$ be the set of all constant
functions plus the identity function (the
latter must be adjoined to make $\leq_\F$ be a preorder), since in this case 
for different
$A,B \subseteq \RR$ we can have $A \leq_\F B$ only if $A = \RR$ and $B
\neq \emptyset$ or $A = \emptyset$ and $B \neq \RR$ (in all the other
cases we have $A \nleq_\F B$). 
Thus we get a degree-structure which is not 
very interesting, and the reason is basically
that we have few functions to reduce one set to another: we can
avoid this unpleasant situation requiring that $\F$ contains a
very simple but sufficiently rich set of functions, such as the
set $\L$ (note that this condition already implies $\id \in \F$). 
On the other hand, we want also to avoid that $\F$ contains too many
functions. In fact, if for example we consider the set $\F$ of
\emph{all} the functions from $\RR$ to $\RR$ we have a lot of functions at our disposal,  and we can
reduce every
set to any other (except for $\emptyset$ and $\RR$), therefore we get a finite and trivial structure of $\F$-degrees (the same structure can be obtained considering any set $\F$ which
contains all the
two-valued functions).
To avoid this situation we can require that
$\F$ is not too large,
i.e.\ that there is a surjection
$j\colon\RR \onto \F$.\label{smallness}
All these considerations naturally lead to the following definition.

\begin{defin} \label{defsetofreductions}\index{Set of reductions}
A set of functions $\F \subseteq {}^\RR \RR$ is a \emph{set of
  reductions} if it is closed under composition, $\F \supseteq \L$
and there is a
surjection $j\colon\RR \onto \F$.
\end{defin}

Classical examples of sets of reductions are $\L$, $\W$,  $\sf{D}_\xi$, 
$\Bor$ and so on.
It is interesting to note that this simple definition allows to
describe almost completely the structure of the $\F$-degrees,
 as it is shown in the next Theorem.

\begin{theorem} [\SLOL+\NFS+\DCR]\label{theorgeneralproperties}
Let $\F \subseteq {}^\RR \RR$ be a set of reductions. Then
\begin{enumerate}[i)]
\item $\leq_\F$ is a well-founded preorder on $\pow(\RR)$;
\item there is no $\leq_\F$-largest set and $\leng(\leq_\F)= \Theta$;
\item anti-chains have length at most $2$ and are of the form
  $\{A,\neg A\}$ for some set $A$;
\item $\RR \nleq_\F \neg \RR = \emptyset$ and if $A \neq \emptyset,
  \RR$ then $\emptyset, \RR <_\F A$;
\item if $A \nleq_\F \neg A$ then $A \oplus \neg A$ is $\F$-selfudal
  and is the successor of both $A$ and $\neg A$. In particular, after an $\F$-nonselfdual pair
  there is a single $\F$-selfdual degree;
\item if $A_0 <_\F A_1 <_\F \dotsc$ is an $\F$-chain of subsets of
  $\RR$ then $\bigoplus_n A_n$ is $\F$-selfdual and is
  the supremum of these
  sets. In particular if $[A]_\F$ is
  limit of countable cofinality then $A \leq_\F \neg A$;
\item if $A \nleq_\F \neg A$ and $\G \subseteq \F$ is another set of reductions
 then $[A]_\F = [A]_\G$. In particular, $[A]_\F = [A]_\L$.
\end{enumerate}
\end{theorem}

\begin{proof}

Part \textit{iv)} follows from the fact that $\F$ contains all the constant functions (which
  are in $\L$). For \textit{ii)} note that since $\L \subseteq \F$ and we have 
  $\SLOL$, we can also
  assume $\SLO^\F$ by Lemma \ref{lemmabasic}. Thus, using the
  surjection $j: \RR \onto \F$, we can define the Solovay's jump operator $J_\F$ and
  use  Lemma \ref{lemmajumpoperator} to get the result with the standard argument (see e.g.\ Theorem 2.7 in \cite{andrettaslo}).
Part \textit{iii)} immediately follows from $\SLO^\F$. For part \textit{vi)} and 
part \textit{i)}, use the fact the each strict inequality with respect to $\F$-reductions
can be converted in a strict inequality with respect to $\L$-reductions
by $\SLOL$ and Lemma \ref{lemmabasic}: therefore the (proper) chain in part \textit{vi)} 
can be converted in a (proper) chain with respect to $\L$ (and this gives that
$\bigoplus_n A_n$ is selfdual and the supremum of the chain), while
in part \textit{i)} any descending $\F$-chain can be converted in a descending $\L$-chain
(the fact that the non-existence of a descending $\F$-chain implies well-foundness 
can be proved as in \cite{andrettamartin}, using  $\DCR$ and the
  surjection $j \colon \RR \onto \F$). For part \textit{v)},
observe that it can not be the case that $A \leq_\L \neg A$ (otherwise we should have also
$A \leq_\F \neg A$), and therefore $A \oplus \neg A$ is selfdual and is the immediate 
successor of $A$ and $\neg A$. Finally, for part \textit{vii)} we clearly
have $[A]_\G \subseteq [A]_\F$. Towards a contradiction,
  assume that $B$ is not $\G$-equivalent to $A$ but $B \in
  [A]_\F$. Note that $B$ can not be $\G$-equivalent to $\neg A$
  (otherwise, $A
  \equiv_\F \neg A$), hence we have only two cases (we sistematically use
part \textit{v)}): if $B <_\G A$ then we would have that $B \leq_\G B \oplus \neg B \leq_\G
  A$. But then $A \equiv_\F B \oplus \neg B$ would be
  $\F$-selfdual, a contradiction! If $A <_\G B$ simply argue as above but replacing the role of $B
  \oplus \neg B$ with $A \oplus \neg A$ to get the same contradiction! 
\end{proof}

It is useful to observe that in the previous
Theorem part
\textit{iv)} is provable in $\ZF$, while parts
\textit{ii)}-\textit{iii)} are
true under $\SLO^\F$ alone.
There are essentially two points left open by Theorem \ref{theorgeneralproperties} in the
description of the hierarchy of the $\F$-degrees, namely:
\begin{description}
\item[Question 2] What happens at limit levels of uncountable
cofinality?

\item[Question 3] What happens after a selfdual degree?
\end{description}

In order to answer these two Questions, we first introduce some useful definitions and an important
tool strictly related to the set of
reductions $\F$.

From now on any pointclass closed under $\L$-preimages will be
called \emph{\L-pointclass} (note that, in particular, any
boldface pointclass is an $\L$-pointclass). More generally, if $\F$ is
any \emph{set of functions}, any pointclass closed under $\F$-preimages will
be called \emph{$\F$-pointclass}
(and clearly if
$\F \subseteq \G$ then every $\G$-pointclass is also an $\F$-pointclass).

Moreover if $\Gamma$ is any pointclass we will say that  $f\colon \RR \to \RR$ is
a \emph{$\Gamma$-function} if
$f^{-1}(D) \in \Gamma$ for every $D \in \Gamma$ (this is clearly a natural generalization of
the notion of $\bDelta^0_\xi$-function
introduced before).

\begin{defin}
We will call \emph{characteristic set of
    $\F \subseteq {}^{\RR} \RR$} the pointclass 
\[ \Delta_\F =\{ A \subseteq \RR \mid  A \leq_\F \bN_{\langle 0
  \rangle}\}.\]
\end{defin}

Note if $\F$ is closed under composition $\Delta_\F$ is always an $\F$-pointclass, and
that every $f \in \F$ is automatically a $\Delta_\F$-function (as we will see in some examples in
Section \ref{sectionBAR}, the converse is not always true even if 
we assume that $\F$ is a set of reduction). Moreover, if $\L \subseteq \F$ then $\Delta_\F$ is selfdual 
and is also an $\L$-pointclass.

\begin{defin}
A set of functions $\F$ is \emph{saturated} if for every $f\colon \RR
\to \RR$
\[ f \in \F \iff f \text{ is a }\Delta_\F\text{-function.}\]
\end{defin}
 We will  call the set $\Sat(\F) = \{ f \in {}^\RR \RR \mid  f \text{ is
   a }\Delta_\F\text{-function}\}$ the \emph{saturation} of
 $\F$.
Clearly if $\F$ is closed under composition $\F \subseteq \Sat(\F)$, and $\F = \Sat(\F)$ just
 in case $\F$ is saturated. Moreover if $\id \in \F$ then $\Delta_{\Sat(\F)} =
 \Delta_\F$. In fact, $\Delta_\F \subseteq \Delta_{\Sat(\F)}$ by the
 previous observation. Conversely, let $A \in \Delta_{\Sat(\F)}$. By
 definition there is some $f \in \Sat(\F)$ such that $A =
 f^{-1}(\bN_{\langle 0\rangle})$, and since $f$ is a
 $\Delta_\F$-function and $\bN_{\langle 0 \rangle} \in \Delta_\F$ we
 have also $A \in \Delta_\F$. Thus $\Sat(\F)$ is a maximum (with respect
 to inclusion) among those sets of reductions $\G$ such that
 $\Delta_\G = \Delta_\F$.

\begin{remark}\label{rem3.2}
Assume that $\F$ is a set of functions 
(but not necessarily a set of reductions). Assuming $\SLOL$ and that each $\bN_s$ is in
$\Delta_\F$, we get that if
$\Delta_\F$ is \emph{bounded} in the Lipschitz hierarchy (i.e.\ if there is some
$B \subseteq \RR$ such that $D \leq_\L B$ for every $D \in \Delta_\F$) then
there is a surjection $j \colon  \RR \onto \F$. In fact, we can take the bound $B$ to
be such that $B \nleq_\W \neg B$ and
$\bGamma = \{ A \subseteq \RR \mid A \leq_\L
B\}$ is a boldface pointclass which contains all the countable intersections
of sets in $\Delta_\F$ (so that, in particular, we
have that $\bGamma$ admits a universal set $U \subseteq \RR \times \RR$ --- see e.g.\ 
Theorem 3.1 in \cite{andrettaslo}). Since
\[ (x,y) \in {\rm graph}(f) \iff \forall s \in \seqo (y \in \bN_s \imp x \in f^{-1}
(\bN_s)),\]
we have that the graph of any $f \in \F$
 is in $\bGamma$. 
Now let $f_0$ be any fixed function in $\F$ and for every $x \in \RR$ let
$j(x) = f$ if $h^{-1}(U_x) = {\rm graph}(f)$, and $j(x) = f_0$ otherwise: 
it is not hard to check that $j$ is the surjection
required.

Conversely, assume that $\Delta_\F$ is unbounded (i.e.\ ``cofinal'') in
the Lipschitz hierarchy
and that there is a surjection $j \colon  \RR \onto \F$. Let $i$ be a
surjection of $\RR$ onto $\L$ and for every $x,y \in \RR$ put $f_x = j(x)$ and
$l_y = i(y)$. Now assume $\SLOL+\NFS+\DCR$ and define
$s \colon  \RR \to \Theta \colon  x \oplus y \mapsto
\rank{l_y^{-1}(f_x^{-1}(\bN_{\langle 0 \rangle}))}_\L$.
It is easy to check that $s$ is onto, contradicting the definition of $\Theta$.

Therefore, under $\SLOL+\NFS+\DCR$ any set of functions $\F$ closed under
composition and such that $\bN_s \in \Delta_\F$ for every $s \in \seqo$
admits a surjection $j\colon \RR \onto \F$ if and only
if $\Delta_\F$ is bounded in the
Lipschitz (equivalently, Wadge)
hierarchy. In particular, if $\F \supseteq \L$ and $\F$ satisfies
the
previous conditions, then it is a set of reductions if and only if $\Delta_\F
\neq \pow(\RR)$ (hence the unique saturated set of functions closed under
composition and  which contains $\L$ but is not a set of
reductions is the set of \emph{all} the functions
from $\RR$ to $\RR$). This means that the condition of ``smallness'' on 
the sets
of functions presented at page \pageref{smallness} could be reformulated as
``$\Delta_\F$ bounded in the Lipschitz hierarchy'' or simply ``$\Delta_\F \neq 
\pow(\RR)$''.
\end{remark}

Now we return to consider \emph{sets of reductions}.
It is clear  that  $\F \subseteq \G \imp \Delta_\F \subseteq
\Delta_{\G}$, but the converse is not always true, as we will see in
some examples in Section \ref{sectionBAR}.
Nevertheless, there is a noteworthy situation in
which $\Delta_\F \subseteq \Delta_\G$ implies $\F \subseteq \G$,
namely when $\G=\sf{D}_\xi$ (for some
nonzero $\xi<\omega_1$), although in this case we must
also require that $\bN_s \in \Delta_\F$ for every $s \in
\seqo$ (we will see that this is always the case if
$\F$ is a \emph{Borel-amenable} set of reductions). First notice that
$\Delta_{\sf{D}_\xi} = \bDelta^0_\xi$ (and $\Delta_\Bor = \bDelta^1_1$). 

\begin{proposition}\label{propDeltaxilevels}
  Let $\F$ be a set of reductions such that
each $\bN_s$ is in $\Delta _\F$, and let
$\xi < \omega _1$ be a nonzero ordinal.
If $\Delta_\F \subseteq \bDelta^0_\xi$
then $\F \subseteq \sf{D}_\xi$. 
\end{proposition}

\begin{proof}
  Since $\sf{D}_\xi \subseteq \sf{D}_{\xi'}\) for $\xi < \xi'$,
we may assume that $\xi$ is least such that
$\Delta_\F \subseteq \bDelta^0_\xi$.
Suppose $\xi  > 1$ and let $f \in \mathcal{F}$: by definition
of $\sf{D}_\xi$ we must show that if $A$ is
$\bSigma^0_\xi$, then so is $f ^{-1} ( A )$.
Choose $A_n \in \bPi^0_{\nu_n}$ with $\nu_n < \xi$
such that $A =  \bigcup_{n \in \omega} A_n$.
By Borel-determinacy either $\bDelta^0_{\nu_n} \subseteq
\Delta_\F$ or $\Delta_\F \subseteq \bDelta^0_{\nu_n}$,
but the latter cannot hold by the minimality of $\xi$, hence
$\bDelta^0_{\nu_n} \subsetneq
\Delta_\mathcal{F}$ and therefore $\bPi^0_{\nu_n}
\subsetneq \Delta_\F$ for each $n$.
As $f$ is a $\Delta_\F$ function, $f^{-1} ( A_n )
\in \Delta_\F \subseteq \bDelta^0_\xi$,
hence $f^{-1} ( A ) \in \bSigma^0_\xi$ as required.

The case $\xi = 1$ is trivial, and it is left to the reader. 
\end{proof}

Notice that the same result is trivially true if we consider Borel functions instead of
$\bDelta^0_\xi$-functions.

\section{Borel-amenability}\label{sectionborelamenability}

A.\ Andretta and D.\ A.\ Martin
proposed in \cite{andrettamartin}  a notion of \emph{amenable} 
set of functions essentially
adjoining the following condition to 
Definition \ref{defsetofreductions}: \emph{for every countable
family $\{f_n \mid n \in \omega\} \subseteq \F$ we have that
$\overline{\bigoplus}_n f_n \in \F$, where \[ \overline{\bigoplus\nolimits}_n f_n(x) =
f_{x(0)}(x^-)\] and $x^- = \langle x(n+1) \mid n \in \omega
\rangle$.}
This condition can be recast in a different way.

\begin{proposition}\label{propequivconditions}
  The following are equivalent:
  \begin{enumerate}[i)]
    \item if $\{f_k \mid k \in \omega\}
    \subseteq \F$ then $\overline{\bigoplus}_k f_k \in \F$;
    \item if $\{f_k \mid k \in \omega\} \subseteq \F$ then
    $\bigcup_k (\restr{f_k}{\bN_{\langle k \rangle}}) \in \F$ and
    $\Lip \subseteq \F$.
  \end{enumerate}
\end{proposition}

\begin{proof}
To see \textit{ii)} $\imp$ \textit{i)}, let
$\{f_k \mid  k \in \omega \} \subseteq \F$ and define $f^-_k(x) = (f_k(x))^-$ for every $x \in \RR$ and $k \in \omega$: then clearly $f^-_k \in \F$ (since $\Lip \subseteq \F$) and
$\overline{\bigoplus}_k f_k = \bigcup_k
(\restr{f^-_k}{\bN_{\langle k \rangle}}) \in \F$. 
To prove  \textit{i)} $\imp$ \textit{ii)}, let $\{f_k \mid k \in \omega \} \subseteq \F$ and 
for every $k \in
\omega$ put
\[ f^+_k\colon  \RR \to \RR \colon  x \mapsto f_k(k \conc x).\]
Since for every $k
\in \omega$ the function $x \mapsto k \conc x$ is in $\L \subseteq
\F$, we get $f^+_k \in \F$ and hence
$\overline{\bigoplus}_k f^+_k=\bigcup_k(\restr{f_k}{\bN_{\langle k \rangle}}) \in \F$. 
Let now $f \in \Lip$, and let $n
\in \omega$ be smallest 
 such that $f \in \Lip(2^n)$: we will prove by
induction on $n$ that $f \in \F$. If $n=0$ then $f \in \L \subseteq
\F$. Now assume that $\Lip(2^n) \subseteq \F$ and pick any $f
\in \Lip(2^{n+1})$. For every $k \in \omega$ define
$f_k \colon  \RR \to \RR\colon x \mapsto f(k \conc x)$.
Since for every $x,y \in \RR$
\begin{multline*}
d(f_k(x),f_k(y)) = d(f(k \conc x),f(k \conc y)) \leq \\ \leq 2^{n+1} d(k
\conc x,k \conc y)= 2^{n+1} \cdot \frac{1}{2} d(x,y) = 2^n d(x,y),
\end{multline*}
we have $\{f_k \mid k \in \omega\} \subseteq \Lip(2^n) \subseteq \F$,
and thus $f=\overline{\bigoplus}_k f_k \in \F$ by our hypotheses.
\end{proof}

Roughly speaking, this condition of amenability says that ($\Lip \subseteq
\F$ and) if we have a
``simple''
partition of $\RR$, i.e.\ composed by the simplest (in the sense
of $\leq_\L$) nontrivial sets (namely, sets
in $\Delta_\L$: every $\Delta_\L$-partition is always refined by $\langle \bN_{\langle
k \rangle} \mid  k \in \omega \rangle$, so  our
definition
based only on sets of the form $\bN_{\langle k \rangle}$  is not
restrictive), and we use on each piece of the
partition a function from $\F$ (as complex as we want), the
resulting function is already in $\F$. But the simplest
sets from ``the point of view'' of $\F$ are those in $\Delta_\F$,
hence it seems quite natural to extend the  definition of ``amenable'' to the
following:

\begin{defin} \label{defBorelamenable}
A set of reductions $\F\subseteq \Bor$  is
\emph{Borel-amenable} if:
\begin{enumerate}
\item $\Lip \subseteq \F$;

\item for every $\Delta_\F$-partition $\langle D_n \mid n \in \omega
  \rangle$ and every collection $\{ f_n\mid n \in \omega\}$ of
  functions from $\F$ we have that
\[ f = \bigcup_{n\in \omega} (\restr{f_n}{D_n}) \in \F.\]
\end{enumerate}

We will denote by $\sf{BAR}$ the set of all the Borel-amenable
sets of reductions.
\end{defin}

\begin{remark}
  The condition  $\F \subseteq \Bor$, as
  already observed, can be recast in an equivalent way
 by requiring that
  $\Delta_\F  \subseteq \bDelta^1_1$.
Note also that
  this condition  already implies that there is a surjection
    $j\colon \RR \onto \F$ (as observed in Remark \ref{rem3.2}), while the condition $\Lip \subseteq \F$ implies that
    $\L \subseteq \F$: hence if a \emph{set of functions} $\F\subseteq
    \Bor$  satisfies the
  two conditions in Definition \ref{defBorelamenable} and is closed under composition,
then it is automatically a Borel-amenable \emph{set of reductions}.
\end{remark}

Almost all the sets of functions one is willing to use as
reductions, such as continuous functions,
$\bDelta^0_\xi$-functions and
Borel functions, are examples of Borel-amenable sets of reductions.
In particular, by Proposition \ref{propequivconditions}, a
Borel-amenable set of reductions $\F$ is always closed under the operation
$\overline{\bigoplus}_n$, and since $\bDelta^0_1$
is the smallest $\L$-pointclass closed under $\bigoplus_n$, 
it is easy to check that $\bDelta^0_1 \subseteq \Delta_\F$ 
(hence, in particular, $\bN_s
\in  \Delta_\F$
for every $s \in {}^{<\omega}\omega$). Moreover,
if $\F \in \sf{BAR}$ then parts
\textit{ii)}-\textit{vi)} of Theorem \ref{theorgeneralproperties}
are true under $\SLO^\F$ alone (see Lemma 3 of \cite{andrettamartin} for
parts \textit{v)} and \textit{vi)}), and  if $\F,\G \in \sf{BAR}$ 
then the first part of \textit{vii)} is
provable also under $\SLO^\G$. 
Finally, we can also
establish a minimum among those Borel-amenable sets of reductions with
the same characteristic set. In fact, given a Borel-amenable set of reductions
$\F$, let $\F^\Lip$ be the collection of the functions locally in $\Lip$ on a
$\Delta_\F$-partition. Then $\F^\Lip \subseteq \F$ since $\F$ must satisfy condition
$2$ of the definition of Borel-amenability and $\Lip \subseteq
\F$.
This implies also $\Delta_{\F^\Lip} \subseteq \Delta_\F$.
Conversely, $\Delta_\F \subseteq \Delta_{\F^\Lip}$ since if $D \in
\Delta_\F$ the function $f_D = (\restr{g_0}{D}) \cup (\restr{g_1}{\neg D})$, where
$g_0,g_1 \in \L \subseteq \F$ are the constant functions with value, respectively,
$\vec{0}$ and $\vec{1}$, is in $\F^\Lip$ and reduces $D$ to $\bN_{\langle 0 \rangle}$. Thus
if $\G$ is such that $\Delta_\G = \Delta_\F$ then $\F^\Lip =
\G^\Lip \subseteq \G$. In particular, this implies that $\sf{D}_1^\Lip$
(the set of all the functions locally in $\Lip$ on a clopen partition) is a
subset of \emph{any} Borel-amenable set of reductions.
A similar argument allow us to prove  the following result.

\begin{proposition}\label{proponlyDxi}
  Let $\F$ be a Borel-amenable set of reductions. Then either
  $\Delta_\F = \bDelta^1_1$ or there is some nonzero $\xi< \omega_1$
  such that $\Delta_\F=\bDelta^0_\xi$.
\end{proposition}

\begin{proof}
  Assume that $\Delta_\F \subsetneq \bDelta^1_1$ and let $\xi <
  \omega_1$ be the smallest nonzero ordinal such that $\Delta_\F
  \subseteq \bDelta^0_\xi$. If $D \in \bDelta^0_\xi$, then there is
  some partition $\langle D_n \mid n \in \omega \rangle$ of $\RR$ such that $D =
  \bigcup_{i \in I} D_i$ for some $I \subseteq \omega$ and $D_n \in
  \bPi^0_{\mu_n}$ for some $\mu_n < \xi$ (see Theorem 4.2 in 
\cite{mottorosnewcharacterization}). Since $\bDelta^0_\mu
  \subsetneq \Delta_\F$ for every $\mu < \xi$ (by minimality of
  $\xi$), we have $\{D_n \mid n\in \omega\} \subseteq \bigcup_{\mu<\xi} \bPi^0_\mu
  \subseteq \Delta_\F$. Let $g_0,g_1$ be defined as above, 
  and put
  $f_i = g_0$ if $i \in I$ and $f_i=g_1$ otherwise. By
  Borel-amenability,
$f= \bigcup_{n \in \omega}(\restr{f_n}{D_n}) \in \F$ and $f^{-1}(\bN_{\langle 0 \rangle}) = D$, 
i.e.\ $D \in \Delta_\F$: therefore $\bDelta^0_\xi \subseteq \Delta_\F$.
\end{proof}

Notice that this Proposition easily implies Proposition
\ref{propDeltaxilevels} (in the special case $\F \in \sf{BAR}$) 
and that $\Delta_\F$ is always an algebra of sets, i.e.\
that it is closed under complements, finite intersections and finite
unions (this fact directly follows also from Lemma
\ref{lemmaintersection}). Moreover, as a corollary one gets that either ${\rm
  Sat}(\F) = \Bor$ or ${\rm Sat}(\F) = \sf{D}_\xi$ for some nonzero
$\xi<\omega_1$.
\\

In the following couple of Lemmas we will always assume
$\F \in \sf{BAR}$. They are analogous of Lemma 12 and, respectively, Lemma 13 and Proposition 18 of \cite{andrettamartin}, and can be proved in almost the same way (here it is enough to use the second condition of Definition \ref{defBorelamenable}).

\begin{lemma} \label{lemmaintersection}
Let $D \subseteq D'$ be two sets in $\Delta_\F$. For every $A
\subseteq \RR$, if $A \cap D' \neq \RR$ then $A \cap D \leq_\F A \cap
D'$. In particular, if $D \in \Delta_\F$ and $A \neq \RR$ then $A \cap
D \leq_\F A$.
\end{lemma}

\begin{lemma} \label{lemmapartition}
Let $\langle D_n \mid  n \in \omega \rangle$ be a
$\Delta_\F$-partition
of $\RR$ and let $A \neq \RR$.
\begin{enumerate}[a)]
\item If $C \subseteq \RR$ and $A \cap D_n \leq_\F C$ for every $n \in
  \omega$ then $A \leq_\F C$.
\item Assume $\SLO^\F$. If $\forall n \in \omega (A \cap D_n <_\F A)$ then $A
  \leq_\F \neg
A$. Moreover, if $D_n = \emptyset$ for all but finitely many $n$'s then $A$ is not limit.

\item Assume $\SLO^\F$. If $A \leq_\F \neg A$ and $[A]_\F$ is
  immediately above a
nonselfdual pair $\{[C]_\F,[\neg C]_\F\}$ with $C \neq \emptyset,
\RR$, then there is 
$D \in \Delta_\F$ such that $A \cap D, A \cap \neg D
<_\F A$.
\end{enumerate}
\end{lemma}

\begin{defin}
Let $\F \in \sf{BAR}$ and $A \subseteq \RR$. 
We say that $A$ has the \emph{decomposition property with respect to
$\F$} if there is a
$\Delta_\F$-partition $\langle D_n \mid n \in \omega \rangle$ of $\RR$ such
that $D_n \cap A <_\F A$ for every $n$.

Moreover, we will say that $\F$ has the \emph{decomposition property}
(\textbf{DP} for short) if
every $\F$-selfdual set $A \notin \Delta_\F$ has the decomposition
property with respect to $\F$.
\end{defin}

Note that the property \textbf{DP} is essentially a converse to part \textit{b)}
of Lemma \ref{lemmapartition}. The following Theorem is analogous to Corollary 17 of \cite{andrettamartin}.

\begin{theorem}\label{theorstructure}
Let $\F$ be a Borel-amenable set of reductions with the {\rm
\textbf{DP}}. Then
\begin{enumerate}[i)]
\item if $[A]_\F$ is limit of uncountable cofinality with respect
to $\leq_\F$ then
  $A \nleq_\F \neg A$;
\item assume $\SLO^\F$: then after a selfdual
  $\F$-degree there is an $\F$-nonselfdual pair.
\end{enumerate}
\end{theorem}

\begin{proof}
  \begin{enumerate}[i)]
  \item Suppose that $A$ is $\F$-limit of uncountable cofinality 
(hence, in particular, $A \notin
\Delta_\F$) and assume towards a contradiction that $A \leq_\F \neg A$. Let $\seq{D_n}{n\in \omega}$ be a
$\Delta_\F$-partition
 of $\RR$ such that $D_n \cap A <_\F A$ for each $n$ (which exists
 since  $\F$ has
the \textbf{DP}). If $B \subseteq \RR$ is such that $A \cap D_n
\leq_\F B$ for every $n$ then $A \leq_\F B$ by Lemma
\ref{lemmapartition}:
hence $A$ is the supremum of the family $\mathcal{A} = \{ A \cap D_n \mid
n\in\omega\}$ and therefore is of countable cofinality, a contradiction!

\item It is enough to prove that if $A$ and $B$ are $\F$-selfdual and
  $A <_\F  B$ (which in particular implies $B \notin \Delta_\F$) then
$B$ is not the successor
of  $A$. Using \textbf{DP},
let $\seq{D_n}{n \in \omega}$ be a $\Delta_\F$-partition such that $D_n \cap B
<_\F B$  for every $n$.
If $D_n \cap B \leq_\F A$ for each $n\in \omega$, then $B \leq_\F A$ by Lemma
\ref{lemmapartition}, a contradiction! Thus there
is some
$n_0\in\omega$ such that $D_{n_0} \cap B \nleq_\F A$: hence $\SLO^\F$ imples
$\neg A  \leq_\F
D_{n_0} \cap B$, and since $A \leq_\F \neg A$ we get $A<_\F D_{n_0}
\cap B  <_\F B$. \qedhere
  \end{enumerate}
\end{proof}

This proves that we can answer Question 1 and Question 2 for every $\F \in \sf{BAR}$ which 
satisfies \textbf{DP}. We will see in the next Section
that, fortunately, \emph{every} Borel-amenable set of reductions has this 
property, thus under $\SLOL+\NFS+\DCR$, we can 
completely determine the degree-structure induced by any reasonable
(i.e.\ Borel-amenable) set of reductions.
However we first want to go further and show that if $\F$ is as in the
previous Proposition then the structure of the $\F$-degrees is
completely determined by the set $\Delta_\F$. This is the reason for which
the set $\Delta_\F$ has been called ``characteristic set''.

\begin{theorem}[\SLOL]\label{theorequivalentreductions}
Let $\F, \F' \in \sf{BAR}$ be such that
$\Delta_\F \subseteq \Delta_{\F'}$ and suppose that $\F$ has the {\rm
\textbf{DP}}. Then for every $A,B \subseteq \RR$
\[A \leq_\F B \imp A \leq_{\F'} B.\]
In particular, if $\F \in \sf{BAR}$ has the
{\rm \textbf{DP}} and $\F' \subseteq \F$ is a Borel-amenable
set of reductions such that $\Delta_\F =
\Delta_{\F'}$, then for every $A,B \subseteq \RR$
\[ A \leq_\F B \iff A \leq_{\F'} B.\]
\end{theorem}

\begin{proof}
We must take cases.
If $A<_\F
B$ then $A <_\L B$ by Lemma \ref{lemmabasic} and hence, in particular,
$A \leq_{\F'} B$: thus we can assume $A \equiv_\F B$ for the other cases. If $A \nleq_\F
\neg A$ then $[A]_\F =[A]_\L$ by part \textit{vii)} of
Theorem \ref{theorgeneralproperties} and, since $B \in [A]_\F$ by our
assumption, we have also $A \equiv_\L B$: thus $A
\leq_{\F'} B$. If $A \in \Delta_\F$ then also $B \in \Delta_\F$,
and hence $A \leq_{\F'} B$ since $\Delta_\F \subseteq \Delta_{\F'}$.
Therefore it remains only to consider the case $A
\equiv_\F \neg A \equiv_\F B \notin \Delta_\F$. 
Since $\F$ has the decomposition property,
there is some $\Delta_\F$-partition $\langle D_n \mid n\in \omega \rangle$ such
that $D_n \cap A <_\F A \equiv_\F B$ for every $n$. In particular,
using Lemma \ref{lemmabasic} and $\SLOL$, we have that $D_n \cap A
<_\L B$ and
hence, since $\F'$ is Borel-amenable and $\Delta_\F \subseteq \Delta_{\F'}$,
$A \leq_{\F'} B$ by Lemma \ref{lemmapartition}.
\end{proof}

Note that under the assumption $\F' \subseteq \F$ we can reprove the same
result assuming only $\SLO^{\F'}$
instead of $\SLOL$. The previous Theorem
allow also to compare different sets of reductions
in term of the hierarchy of degrees induced by them: let us
say that two sets of reductions $\F$ and $\G$ are
\emph{equivalent} ($\F \simeq \G$ in symbols) if they induce the same
hierarchy of degrees, that is if for every
$A,B \subseteq \RR$ we have $A \leq_\F B$ if and only if $A \leq_\G B$.
Then Theorem \ref{theorequivalentreductions} implies that if
$\F$ and $\G$ are two Borel-amenable sets of reductions (with the
\textbf{DP})  we have
\[ \F \simeq \G \iff \Delta_\F = \Delta_\G.\]

\section{The Decomposition Property}\label{sectionDP}

In this Section we will prove that every
Borel-amenable set of reductions $\F$ has the decomposition
property, but we first need the following two Lemmas, which are refinements of
Theorem 13.1 and Theorem 13.11 in \cite{kechris}. Lemma
\ref{lemmarelativization} 
is a simple variation (for the Baire space endowed with the usual topology) 
of  Theorem 22.18
of \cite{kechris} (stated there, although in a slightly different form, as 
Exercise 22.20), while Lemma
 \ref{lemmafunction}  follows from Lemma \ref{lemmarelativization} by
 standard arguments. However we want to point out that these results
 could be  generalized (with 
slightly different proofs) by considering \emph{any} $\L$-pointclass
$\Delta \subseteq \bDelta^1_1(\RR,\tau)$ --- see \cite{mottorosthesis}.
This observation allows also to generalize Theorem \ref{theorDP} to
almost all the sets of reductions $\F \subseteq \Bor$ (not only to the
Borel-amenable ones), but we will not
use this fact here.

\begin{lemma}\label{lemmarelativization}
Let $d$ be the usual metric on $\RR$, $\tau$ the topology induced by $d$, 
and $\xi$ be any nonzero
countable ordinal. Let $\Delta$ be either $\bDelta^0_\xi(\RR,\tau)$ or
$\bDelta^1_1(\RR,\tau)$. For any family
$\{D_n \mid  n\in \omega\} \subseteq \Delta$ there is a metric $d'$ 
on $\RR$ such that
\begin{enumerate}[i)]
\item $(\RR,\tau')$ is Polish and zero-dimensional, where $\tau'$ is the topology
induced by $d'$;

\item $\tau'$  refines  $\tau$;

\item each $D_n$ is $\tau'$-clopen;

\item there is a countable clopen basis $\B'$ for $\tau'$ such
that $\B'
  \subseteq \Delta$.
\end{enumerate}
\end{lemma}

\begin{lemma} \label{lemmafunction}
Let $d$, $\tau$, $\xi$ and $\Delta$ be as in the previous Lemma . Moreover,
let $\tau' \supseteq \tau$ be any zero-dimensional
 Polish topology on $\RR$ which admits a countable clopen basis $\B' \subseteq \Delta$. For any
$\Delta$-function  $f\colon \RR \to
\RR$ there is a metric $d_f$ on
$\RR$ such that
\begin{enumerate}[i)]
\item $(\RR,\tau_f)$ is Polish and zero-dimensional, where $\tau_f$ is the topology 
induced by $d_f$;
\item $\tau_f$ refines $\tau'$;
\item there is a countable clopen basis $\B_f$ for $\tau_f$ such that $\B_f
  \subseteq \Delta$;
\item $f \colon  (\RR,\tau_f) \to (\RR,\tau')$ is continuous.
\end{enumerate}
Moreover $d_f$ can be chosen in such a way that condition \textit{iv)} can 
be strengthened to
\begin{enumerate}[iv')]
\item $f \colon  (\RR,\tau_f) \to (\RR, \tau_f)$ is continuous.
\end{enumerate}
\end{lemma}

Now we are ready to prove the main Theorem of this Section, which 
sharpens the argument used to prove Theorem
16 in \cite{andrettamartin}. Since our proof closely follows the original
one, we will only sketch it highlighting the
modification that one has to adopt in this new context.
Therefore the reader interested in a complete proof
should keep a copy of \cite{andrettamartin} on hand and read the
corresponding proofs parallel to one another.

\begin{theorem}[\NFS] \label{theorDP}
Let $\F$ be a Borel-amenable set of reductions. 
Assume that $A \leq_\F \neg A \notin \Delta_\F$. Then
$A$ has the decomposition property with respect to $\F$.
\end{theorem}

\begin{proof}
Suppose towards a
contradiction that for every $\Delta_\F$-partition
$\seq{D_n}{n \in \omega}$ of $\RR$ there is some $n_0 \in \omega$ such that $D_{n_0} \cap
A \equiv_\F A$. The next Claim is quite similar to Claim 16.1 in
\cite{andrettamartin}, but we will completely reprove it here in order to fill
a little gap in the original proof.

\begin{claim}\label{claimpartialreduction}
Let $D \in \Delta_\F$ and assume $A \cap D \equiv_\F A$. Then there is some
$f \in \F$ such that ${\rm range}(f) \subseteq D$ and
\[ \forall x\in D(x \in A \cap D \iff f(x) \in \neg A \cap D).\]
\end{claim}

\begin{proof}[Proof of the Claim]
Note that $D = \emptyset$ and $D \subseteq A$ are forbidden since $D <_\F A$ (and $A \cap D = D$ would contradict
 $A \cap D \equiv_\F A$), while if $D = \RR$ any
$\F$-reduction of $A$ to $\neg A$ will suffice. Thus we can assume $D \neq
\emptyset, \RR$ and $\neg A \cap D \neq \emptyset$. By Lemma
\ref{lemmaintersection} we have $\neg A \cap D \leq_\F \neg A \equiv_\F A
\equiv_\F A \cap D$ ($\neg A$ is nonempty because it is
selfdual). Let $h \in \F$ be such that $h^{-1}(A \cap D) =
\neg A \cap D$ and choose some $y \in \neg A \cap D$. Now put
\[ k(x) =
\begin{cases}
x & \text{if } x \in D\\
y & \text{if } x \notin D.
\end{cases}
\]
Note that $k \in \F$ (since $\F$ is Borel-amenable), and let $f = k \circ
h$. Clearly $f \in \F$ and ${\rm range}(f) \subseteq D$. We will now prove that $x \in \neg A \cap D \iff f(x) \in A \cap D$ for every $x \in D$ (which easily implies the result).  Let $x \in
D$: if $x \in \neg A \cap D$ then $h(x) \in A \cap D \subseteq D$ and
hence also $f(x) \in A \cap D$. Conversely, if $x \in A \cap D$ then
$h(x) \in \neg A \cup \neg D$: if $h(x) \in D$ then $f(x) = h(x) \in
\neg A \cap D$, otherwise $f(x) = y \in \neg A \cap D$ and
in both cases we are done.
\renewcommand{\qedsymbol}{$\square$ \textit{Claim}}
\end{proof}

One must now
construct the following sequences:
\begin{enumerate}[i)]
\item a sequence $\dotsc \subseteq D_1 \subseteq D_0=\RR$ of sets in $\Delta_\F$
such that $A \cap D_n \equiv_\F A$ for every $n \in \omega$
(in particular, $D_n \neq \emptyset$);
\item a sequence of functions $f_n \in \F$ as in the previous Claim,
  i.e.\ such that
\[ \forall x \in D_n (x \in A \cap D_n \iff f_n(x) \in \neg A \cap
D_n);\]
\item a sequence of separable complete metrics $d_n$ on $\RR$ such
that
$d_0$ is the usual metric on $\RR$,
  the topologies $\tau_n$
generated by the metrics $d_n$ are all zero-dimensional,
 $\tau_{n+1}$
refines $\tau_n$, $D_n$ is clopen with respect to $\tau_n$, every
$\tau_n$ admits a countable clopen basis $\B_n 
\subseteq \Delta_\F$, the function $f_n\colon  (\RR,\tau_{n+1})
\to (\RR,\tau_n)$ is continuous, and for every $m \leq n$ and every
$x,y \in D_{n+1}$
\begin{equation}\label{eqCauchy}
\tag{$*$}
d_m(g_m \circ \dotsc \circ g_n(x), g_m \circ \dotsc \circ
g_n(y))<2^{-n},
\end{equation}
where each $g_i$ is either $\restr{f_i}{D_{i+1}}$ or the identity on
$D_{i+1}$. 
\end{enumerate}

Then we can conclude our proof simply replacing
the  $B_n$'s with the $D_n$'s in the original proof (that is we can construct 
a flip-set from the sequences above: this gives the desired contradiction).

The construction of the required sequences is by induction on $n
\in\omega$: set $D_0=\RR$, and let $d_0$ be the usual metric on $\RR$
and $f_0 \in \F$ be any function witnessing $A \leq_\F \neg
A$. Then suppose that $D_m$, $\B_m$, $f_m$ and $d_m$ have been defined for
all $m \leq n$.  

\begin{claim}
  For each $m\leq n$ there is a $\Delta_\F$-partition $\langle C^i_m \mid i\in
\omega \rangle$ of $D_m$ such that $d_m$-${\rm diam}(C^i_m) < 2^{-n}$ and
$C^i_m$ is $\tau_m$-clopen for
every $i \in 
\omega$.
\end{claim}

\begin{proof}[Proof of the Claim]
  Since $\B_m \subseteq \Delta_\F$ is a countable basis for $\tau_m$, we
can clearly find a countably family $\{\hat{C}^i_m \mid i \in \omega\} \subseteq \Delta_\F$
such that $D_m \subseteq \bigcup_{i \in \omega} \hat{C}^i_m$ and
$d_m$-${\rm  diam}
(\hat{C}^i_m) < 2^{-n}$ for every $i \in \omega$. Now simply define
$C^0_m = \hat{C}^0_m \cap D_m$ and $C^{i+1}_m = (\hat{C}^{i+1}_m \cap D_m)
\setminus (\bigcup_{j \leq i} \hat{C}^j_m)$. Since $\Delta_\F$ is an algebra and each
 $\hat{C}^i_m$ is $\tau_m$-clopen, 
$\langle C^i_m \mid i \in \omega \rangle$ is the required
$\Delta_\F$-partition.
\renewcommand{\qedsymbol}{$\square$ \textit{Claim}}
\end{proof}

The inductive step can now be completed  as in the original proof
using the previous Claim and applying Lemma
\ref{lemmafunction}.
\end{proof}

Notice that, as for the Borel case,
the nonexistence of flip-sets is used in a ``local way'' in the
proof
of Theorem \ref{theorDP}: in fact the flip-set obtained is the continuous preimage of $A$ and
therefore the proof only requires that there are no flip-sets $\W$-reducible
to $A$. 
Observe also that this kind or argument (which is based on 
relativizations of topologies) cannot be applied beyond the Borel context. In 
fact, if $X$ and $Y$
are Polish spaces, $f\colon X \to Y$ is a Borel function and $A \subseteq X$ is a
Borel set such that $\restr{f}{A}$ is injective, then also $f(A)$ is Borel 
(see Corollary 15.2 in
\cite{kechris}). Now suppose that $\tau$ and $\tau'$ are two Polish topologies
on $X$ such that $\bDelta^1_1(X,\tau) \subsetneq \bDelta^1_1(X,\tau')$ and let
$A \in \bDelta^1_1(X,\tau') \setminus \bDelta^1_1(X,\tau)$. Applying the
preceding result to $f = \id \colon  (X, \tau') \to (X,\tau)$, 
we should have that $f(A) = \id(A) = A \in \bDelta^1_1(X,\tau)$, a contradiction!
Therefore we can not refine the standard topology $\tau$ of $\RR$ in order
to make
clopen (or even just Borel) a set which was not in $\bDelta^1_1(\RR,\tau)$
without losing the essential condition that the new space is still Polish\footnote{The
author would like to thank A.~Marcone for suggesting the present
argument which considerably simplify a previous proof of this fact.}. 
Nevertheless it is possible to study the hierarchies of degrees induced by
sets of reductions $\F \supsetneq \Bor$ using a different kind of argument
--- see the forthcoming \cite{mottorossuperamenability}.

Theorem \ref{theorDP} (together with Theorem \ref{theorstructure}) completes the description of
the degree-structure induced by
$\leq_\F$ when $\F$ is a Borel-amenable set of reductions,  showing
that the $\F$-structure looks like the structure
of the Wadge degrees:

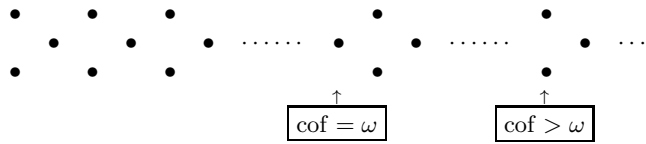
\begin{small}
\begin{equation*}
\begin{array}{llllllllllllll}
\bullet & & \bullet & & \bullet & & & & \bullet & & & \bullet
\\
& \bullet & & \bullet & & \bullet & \cdots \cdots & \bullet
& & \bullet &\cdots\cdots & & \bullet & \cdots
\\
\bullet & & \bullet & & \bullet & & & & \bullet & & & \bullet
\\
& & & & & & &
\, \makebox[0pt]{$\stackrel{\uparrow}{\framebox{${\rm cof} = \omega$}}$}
& & & &
\, \makebox[0pt]{$\stackrel{\uparrow}{\framebox{${\rm cof} > \omega$}}$}
\end{array}
\end{equation*}
\end{small}

Recall also that
$\sf{BAR}$, in 
particular, contains almost 
all the cases already studied, namely continuous functions, 
$\bDelta^0_2$-functions and Borel functions: thus these results provide
an alternative proof for the results about those
degree-structures.
Moreover, we highlight that the principle $\SLOL$ is
needed only
to prove the well-foundness of $\leq_\F$, since all the other results
are provable under $\SLO^\F +\NFS+\DCR$. 

We conclude this Section with the following Corollary which completely
characterize the $\F$-selfdual degrees.

\begin{corollary}[\SLOL+\NFS+\DCR]\label{corcharacterizeselfdual}
Let $\F$ be a Borel-amenable set of reductions and let $A \subseteq
\RR$ be such that $A \notin \Delta_\F$. Then the following are
equivalent:
\begin{enumerate}[i)]
\item $A \leq_\F \neg A$;
\item $A$ has the decomposition property with respect to $\F$;
\item if $B$ is $\L$-minimal in $[A]_\F$ then $B \leq_\L \neg B$ and
  $B$ is either limit (of countable cofinality) 
or successor of a nonselfdual pair with respect to
 $\leq_\L$.
\end{enumerate}
\end{corollary}

\begin{proof}
That \textit{i)} is equivalent to \textit{ii)} follows directly from
Theorem \ref{theorDP} and Lemma \ref{lemmapartition}, and obviously
\textit{iii)} implies \textit{i)}. It remains to prove that
\textit{i)} implies \textit{iii)}. By Theorem \ref{theorDP} and
Theorem \ref{theorstructure}, we must
distinguish two cases: if $A$ is limit with respect to $\leq_\F$, then
there must be a countable chain $A_0 <_\F A_1 <_\F \dotsc$ such that
$A$ is the supremum of it: but in this case we get $A_0 <_\L A_1 <_\L
\dotsc$ by $\SLOL$ and Lemma \ref{lemmabasic}, 
and therefore it is easy to check that $\bigoplus_n A_n$ is
$\L$-selfdual, is limit in the $\L$-hierarchy and is also $\L$-minimal
in $[A]_\F$. Similarly, if $[A]_\F$ is a successor degree then there
must be some $C \subseteq \RR$ such that $C \nleq_\F \neg C$ and $A
\equiv_\F C \oplus \neg C$: in this case it is easy to check that $C
\oplus \neg C$ is $\L$-selfdual, is $\L$-minimal in $[A]_\F$, and its $\L$-degree is the successor (in the $\L$-hierarchy) 
of the nonselfdual
pair $\{ [C]_\L, [\neg C]_\L\}$.
\end{proof}

\section{The structure of $\sf{BAR}$}\label{sectionBAR}

We now want to study the structure  $\langle \sf{BAR}, \subseteq \rangle$. 
Clearly, as already observed in the previous
Sections, $\sf{D}^\Lip_1$ and $\Bor$ are, respectively, the minimum
and the maximum of this structure.
Let now $\emptyset \neq \mathscr{B} \subseteq \sf{BAR}$ and put 
$\bigwedge \mathscr{B} = \bigcap \mathscr{B}$.
Then $\bigwedge \mathscr{B} \in \sf{BAR}$ and,
by the properties of the intersection, $\bigwedge \mathscr{B}$ is the infimum for $\mathscr{B}$
(with respect to inclusion). Moreover, $\Delta_{\bigwedge\mathscr{B}} =
\bigcap_{\F \in \mathscr{B}} \Delta_\F$. For one direction 
$\Delta_{\bigwedge
  \mathscr{B}} \subseteq \bigcap_{\F \in \mathscr{B}} \Delta_\F$ by
definition: for the converse, let $D \in \bigcap_{\F \in \mathscr{B}}
\Delta_\F$ and let $g_0, g_1 \in \L$ be the constant fuctions with
value, respectively, $\vec{0}$ and $\vec{1}$. By Borel-amenability, $f
= (\restr{g_0}{D})\cup (\restr{g_1}{\neg D}) \in \F$ for every $\F \in
\mathscr{B}$, and $f^{-1}(\bN_{\langle 0 \rangle}) = D$: hence $D
\in \Delta_{\bigwedge \mathscr{B}}$. In particular, by Proposition
\ref{proponlyDxi} we have $\Delta_{\bigwedge
  \mathscr{B}}=\bDelta^0_\xi$, where $\xi =
\min\{\mu\mid \bDelta^0_\mu = \Delta_\F \text{ for some }\F \in
\mathscr{B}\}$.

Conversely, let $\mathscr{C}_\mathscr{B} = \{ \G \in \sf{BAR}\mid \F \subseteq \G
\text{ for every }\F \in \mathscr{B}\}$:
clearly $\mathscr{C}_\mathscr{B} \neq \emptyset$ (since $\Bor \in
\mathscr{C}_\mathscr{B}$), thus we can define
$\bigvee \mathscr{B} = \bigwedge \mathscr{C}_\mathscr{B}=\bigcap
\mathscr{C}_\mathscr{B}$.
Obviously $\bigvee \mathscr{B} \in \sf{BAR}$, and if $\G \in \sf{BAR}$
is such that
$\F \subseteq \G$ for every $\F \in \mathscr{B}$ then $\G \in
\mathscr{C}_\mathscr{B}$ by definition: hence $\bigvee \mathscr{B} \subseteq \G$ and
$\bigvee \mathscr{B}$ is the supremum for $\mathscr{B}$ with respect
to inclusion.
Thus we have proved that $\langle \sf{BAR},\subseteq \rangle$ is a
complete lattice with minimum and maximum.
Note however that, contrarily to $\bigwedge \mathscr{B}$, the supremum
$\bigvee \mathscr{B}$ has been defined in an undirected way and not
starting from the elements of $\mathscr{B}$. To give a direct
construction of $\bigvee \mathscr{B}$,
first consider the map $*$ which sends a generic set of reductions $\Lip
\subseteq \F
\subseteq \Bor$ such that \emph{$\Delta_\F$ is closed under finite 
intersections} (i.e.\ such that $\Delta_\F$ is an algebra) 
to the set
\[
\begin{split}
\F^* =\Bigl\{ \bigcup\nolimits_{n\in \omega} (\restr{f_n}{D_n})\mid\; & f_n \in \F \text{ for
  every }n \in \omega \text{ and}  \\
  & \langle 
D_n \mid n \in \omega \rangle \text{ is a }\Delta_\F\text{-partition
  of }\RR \Bigr\} .
\end{split} \]

It is easy to check that e.g.\ $\Lip^*= \sf{D}^\Lip_1$.

\begin{theorem}\label{theor*map}
  The map $*$ is a surjection on $\sf{BAR}$ such that:
\begin{enumerate}[i)]
\item $*$ is the identity on $\sf{BAR}$, i.e.\ if $\F \in
  \sf{BAR}$ then $\F^* = \F$;
\item $\F^*$ is the minimal Borel-amenable set of reductions (with
respect to
  inclusion)
  which contains $\F$.
\end{enumerate}
\end{theorem}

\begin{proof}
Part \textit{i)}
is obvious, while for part \textit{ii)} it is enough to observe that
if $\G$ is any set of reductions which
contains $\F$ then $\Delta_\F \subseteq \Delta_\G$: hence if $\G$
satisfies the second condition in Definition \ref{defBorelamenable}
then it must contain all the functions from
$\F^*$.
Since if $\Lip \subseteq \F \subseteq \Bor$ then also $\Lip \subseteq \F^* \subseteq \Bor$ (as $\F \subseteq \F^*$), it remains only to show that if $\Delta_\F$ is closed under finite intersection then $\F^*$ is closed under composition and satisfies
  the second condition in Definition \ref{defBorelamenable}. Let $f =
  \bigcup_{n\in \omega}(\restr{f_n}{D_n})$
  and $g = \bigcup_{k\in \omega}(\restr{g_k}{C_k})$ be two functions from
  $\F^*$, and for every $n,k \in \omega$ put $D_{n,k}
  = f_n^{-1}(C_k) \cap D_n$. Since $f_n \in \F$ is a
  $\Delta_\F$-function and $\Delta_\F$ is closed under finite
  intersections we have that $\langle D_{n,k}\mid n,k \in \omega \rangle$ is
  a $\Delta_\F$-partition of $\RR$. Moreover 
  $g_k \circ f_n \in \F$ and $g \circ f = \bigcup_{n,k \in \omega}(\restr{g_k \circ f_n}{D_{n,k}})$,
hence $g \circ f \in \F^*$ by definition.

Let now $\langle D_n\mid n \in \omega \rangle$ be a $\Delta_{\F^*}$-partition
of $\RR$: we claim that it admits a refinement to a
$\Delta_\F$-partition. In fact, fix any $n \in \omega$ and let $g=\bigcup_{k \in \omega}
(\restr{g_k}{C_k}) \in \F^*$ be a reduction of $D_n$ to $\bN_{\langle
  0 \rangle}$: then the sets $g_k^{-1}(\bN_{\langle 0 \rangle}) \cap
C_k$ form a $\Delta_\F$-partition of $D_n$. Thus we can safely assume
that each $D_n$ is in $\Delta_\F$. Let now $\{ f_n\mid n \in \omega\}
\subseteq \F^*$, $\langle D'_{n,k}\mid k \in \omega
\rangle$ be a $\Delta_\F$-partition of $\RR$, and
$\{f_{n,k}\mid k \in \omega\} \subseteq \F$ be such that $f_n =
\bigcup_{k \in \omega} (\restr{f_{n,k}}{D'_{n,k}})$ for every $n$. Clearly the sets
$D_{n,k} = D_n \cap D'_{n,k}$ form a $\Delta_\F$-partition of $\RR$:
hence the function
\[ f= \bigcup_{n \in \omega} (\restr{f_n}{D_n}) = \bigcup_{n,k \in \omega}
(\restr{f_{n,k}}{D_{n,k}}) \]
is in $\F^*$ by definition, and $\F^*$ satisfies the second condition 
of Definition \ref{defBorelamenable}. 
\end{proof}

Let now
$\hat{\mathscr{B}}$ be the closure
under composition of
$\bigcup \mathscr{B}$, and let $\bigvee \mathscr{B}$ be obtained
applying the map $*$ to $\hat{\mathscr{B}}$, i.e.\
$\bigvee \mathscr{B}= (\hat{\mathscr{B}})^*$.
It is not hard to check that $\hat{\mathscr{B}}$ is a set of
reductions such that $\Lip \subseteq \hat{\mathscr{B}} \subseteq
\Bor$, and that $\Delta_{\hat{\mathscr{B}}} = \bigcup_{\F \in \mathscr{B}} \Delta_\F$ 
(to see this use the fact that every function in
$\hat{\mathscr{B}}$ is the composition of a \emph{finite} number of
functions from $\bigcup \mathscr{B}$): 
hence $\Delta_{\hat{\mathscr{B}}}$ is an algebra and 
$\bigvee \mathscr{B} \in \sf{BAR}$ by Theorem
\ref{theor*map}. Moreover, if $\G \in \sf{BAR}$ is such that 
$\F \subseteq \G$ for every $\F \in \mathscr{B}$, then
$\hat{\mathscr{B}} \subseteq \G$ and thus $\bigvee \mathscr{B} \subseteq
\G$ by Theorem \ref{theor*map} again. Therefore $\bigvee \mathscr{B}$
is the supremum  of $\mathscr{B}$ with respect to inclusion.
Contrarily to the infimum case, one can still prove that
$\Delta_{\bigvee
  \mathscr{B}} \supseteq \bigcup_{\F \in
  \mathscr{B}} \Delta_\F$, but in some cases the other inclusion can
fail. In fact, given e.g.\ $\mathscr{B} =
\{ \sf{D}_m\mid m \in\omega\}$, we have that $\bigvee
\mathscr{B}$ is formed by those functions $f$ which are
in $\bigcup \mathscr{B}$ on a $\bDelta^0_\omega$-partition 
 (a collection which is different from $\sf{D}_\omega$, see
later in this Section): thus
\[\bigcup_{\F \in \mathscr{B}}
\Delta_\F = \bigcup_{n \in \omega} \bDelta^0_n \subsetneq \bDelta^0_\omega =
\Delta_{\bigvee \mathscr{B}}.\]
However it is not hard to see that $\Delta_{\bigvee
  \mathscr{B}} = \bDelta^0_\xi$, where $\xi =\sup\{
\mu\mid \bDelta^0_\mu = \Delta_\F \text{ for some }\F \in
\mathscr{B}\}$. Therefore we have $\Delta_{\bigvee\mathscr{B}} =\bigcup_{\F \in
  \mathscr{B}} \Delta_\F$ if and only if there is some $\F \in
\mathscr{B}$ such that $\Delta_\F=\bDelta^0_\xi = \Delta_{\bigvee
\mathscr{B}}$.

Put now $\F \equiv \G$ just in case $\Delta_\F = \Delta_\G$.
If we assume $\SLOL+\NFS+\DCR$ and $\F,\G \in \sf{BAR}$, then
$\F \equiv \G$ if and only if $\F \simeq \G$ 
(by Theorem \ref{theorDP} and Theorem \ref{theorequivalentreductions}),
hence it is quite  natural
to consider the quotient $\sf{BAR}/_\equiv$ together with the relation
$\preceq$ defined by
\[ {[\F]_\equiv \preceq [\G]_\equiv} \iff {\Delta_\F \subseteq
\Delta_\G}\]
(again, assuming $\SLOL+\NFS+\DCR$, we have $[\F]_\equiv \preceq
[\G]_\equiv$ if and only if $A \leq_\F B \imp A \leq_\G
B$ for every $A,B \subseteq \RR$). 
It follows from Proposition \ref{proponlyDxi} that this structure
is a well-founded linear order of
length $\omega_1+1$. Each equivalence class is of the form
$\{ \F \in \sf{BAR}\mid \Delta_\F = \bDelta^0_\xi\}$
for some $1 \leq \xi \leq \omega_1$ (similarly to the case of
the set of functions $\sf{D}_{\omega_1}$, for notational simplicity we
put $\bDelta^0_{\omega_1}=\bDelta^1_1$), and for this reason we will
say that $\F \in
\sf{BAR}$ is \emph{of level $\xi$}
if $\Delta_\F=
\bDelta^0_\xi$. Moreover,  if we consider a single
equivalence class endowed with the inclusion relation, arguing as
before we get again a complete lattice with minimum and maximum
(here the minimum and the maximum are, respectively, $\sf{D}^\Lip_\xi$
and $\sf{D}_\xi$, where $1 \leq \xi \leq \omega_1$ is the level of any of the
sets of reductions in the equivalence class considered).\\

We now want to give some examples of (different) Borel-amenable sets of
reductions, showing at once that each level of $\sf{BAR}$ contains more than
one element and that there are $\F \subsetneq \G \in \sf{BAR}$ such that $\Delta_\F  = \Delta_\G$ (so that, in particular, $\F \neq \Sat(\F)$).
We extend the notation introduced on page
\pageref{proponlyDxi}.

\begin{defin}
Let $\F,\G$ be two sets of reductions. We will denote by $\F^\G$
the set of all the functions which are \emph{in $\G$ on a $\Delta_\F$-partition}.
In particular, for any nonzero ordinal $\xi\leq\omega_1$, we will
denote by $\sf{D}_\xi^\W$ the set of all the functions which are 
\emph{continuous on a
$\bDelta^0_\xi$-partition}.
\end{defin}

\begin{remark}\label{remDWxi}
  One must be cautious and pay attention to the definition of
  $\sf{D}^\W_\xi$, since it must be distincted from the set
\[ \begin{split}
\tilde{\sf{D}}^\W_\xi = \{f \in {}^\RR \RR \mid \;&\text{there is a
}\bDelta^0_\xi\text{-partition } \langle D_n\mid n \in \omega \rangle \text{
  of } \RR\\
&\text{such that
}\restr{f}{D_n}\text{ is continuous for every }n\}.
\end{split} \]
Clearly $\sf{D}^\W_\xi \subseteq \tilde{\sf{D}}^\W_\xi$ for every $\xi
\leq \omega_1$, and if $\xi \leq 2$ then we have also $\sf{D}^\W_\xi=
\tilde{\sf{D}}^\W_\xi$. But if $\xi>2$ then $\sf{D}^\W_\xi \subsetneq
\tilde{\sf{D}}^\W_\xi$.
To see this, put $D = \{ x \in \RR \mid  \forall n \exists m>n
(x(m) \neq 0)\}$. Clearly $D$ and $\neg D$
form a $\bDelta^0_3$-partition of $\RR$. For every $x \in \RR \cup \seqo$ let $N_x =
\{ n < \leng(x) \mid x(n) \neq 0 \}$ and define the function $f\colon \RR
\to \RR$ by letting $f(x) = \langle x(n)\mid n \in N_x \rangle$ if $x \in D$ and $f(x) = \vec{0}$ otherwise.
It is not hard to check that $\restr{f}{D}$ and $\restr{f}{\neg D}$
are continuous, thus $f \in \tilde{\sf{D}}^\W_\xi$ for every $\xi>2$.
 Nevertheless one can
prove that there is no $\xi \leq \omega_1$ such that $f \in \sf{D}^\W_\xi$.
This is a consequence of the following Claim.

\begin{claim}\label{claim6.2.1}
  Let $\langle C_n\mid n \in \omega \rangle$ be \emph{any} partition of $\RR$ and
$\{ f_n\mid n \in \omega \} \subseteq {}^{\RR} \RR$ be 
such that $\restr{f}{C_n}=\restr{f_n}{C_n}$. Then $f_{n_0}$ is not continuous
for some $n_0 \in \omega$.
\end{claim}

\begin{proof}[Proof of the Claim]
  Since the $C_n$'s cover $\RR$, by the Baire Category Theorem there must be
some $n_0\in \omega$ such that $C_{n_0}$ is not nowhere dense, i.e.\ such that
$\bN_s \subseteq
\Cl(C_{n_0})$ for some $s \in \seqo$.
Observe that for every $A \subseteq
\RR$ and every continuous function $g\colon \RR \to \RR$,
if $\restr{g}{A \cap D} = \restr{f}{A \cap D}$ then
$\restr{g}{\Cl(A)\cap D} = \restr{f}{\Cl(A)\cap D}$ (by 
the continuity of $f$ and $g$ on $D$). Now assume,
towards a contradiction, that $f_{n_0}$ is continuous: then
$\restr{f_{n_0}}{\bN_s \cap D}=\restr{f}{\bN_s \cap D}$. For every $k \in
\omega$ put $x_k = s \conc 0^{(k)} \conc \vec{1} \in \bN_s \cap D$ and
$y_k = s \conc 0^{(k)} \conc \vec{2} \in \bN_s \cap D$, and check that
$f_{n_0}(x_k) = f(x_k) = t \conc \vec{1}$ and $f_{n_0}(y_k)=f(y_k)=t
\conc \vec{2}$ for every $k \in \omega$, where $t = \langle s(n) \mid
n \in N_s \rangle$. Since $x_k \rightarrow s \conc \vec{0}$ and $f_{n_0}$
is continuous on the whole $\RR$, we must have $f_{n_0}(s \conc \vec{0})=t \conc
\vec{1}$. Similarly, since $y_k \rightarrow s \conc \vec{0}$, by continuity
of $f_{n_0}$ again we should have $f_{n_0}(s \conc \vec{0}) = t \conc \vec{2}
\neq t \conc \vec{1}$, a contradiction! Thus $f_{n_0}$ can not be continuous
and we are done.
\renewcommand{\qedsymbol}{$\square$ \textit{Claim}}
\end{proof}

\end{remark}

 Observe now that $\sf{D}_1 = \sf{D}^\W_1$ and that, in particular,
 $\sf{D}^\W_\xi \subseteq \sf{D}_\xi$ for any nonzero
 $\xi\leq\omega_1$. By a remarkable  Theorem
of Jayne and Rogers (Theorem 5 in \cite{jaynerogers}) we have that
$\sf{D}_2 =
\sf{D}^\W_2$, and as an obvious corollary one gets also $\sf{D}_2 \simeq \sf{D}^\W_2$.
The Jayne-Rogers Theorem (and its mentioned corollary) were
used in
\cite{andrettamoreonwadge} to observe
that the so-called \emph{backtrack functions}  are exactly (and thus give
the same hierarchy of degrees as) the functions in $\sf{D}_2$:
this allowed to use all the combinatorics arising from
the backtrack game (for a definition of this game see
  \cite{vanwesepthesis} or \cite{andrettamoreonwadge}) for the
study of the $\sf{D}_2$-hierarchy, and  thus it seems desirable
 to find some extension of
the Jayne-Rogers Theorem in order to simplify the study of
$\leq_{\sf{D}_\xi}$ when  $2<\xi<\omega_1$ (this problem
was first posed by Andretta in his \cite{andrettaslo}).
The first obvious generalization is the statement $\sf{D}^\W_\xi =
\sf{D}_\xi$, but this immediately fails for every $2 < \xi \leq \omega_1$
by the
counter-example given in Remark \ref{remDWxi}. A slightly  weaker
generalization (which is not in contrast with this Remark)
leads to
the following Conjecture.

\begin{conj}\label{conj1}
Let $\xi<\omega_1$ be any nonzero ordinal. Then $\sf{D}_\xi =
\tilde{\sf{D}}^\W_\xi$.
\end{conj}

Unfortunately, as Andretta already observed in \cite{andrettaslo}, this Conjecture is not true for any
$\xi \geq \omega$. In fact there is a function $P \colon \RR \to \RR$ (called
Pawlikowski function\footnote{In \cite{cichonmorayne} the Pawlikowski function 
was defined on a space which is homeomorphic to the Cantor space $\Can$, but
it is clear that it can be extended to a function $P$ defined on $\RR$ without
losing the various properties of the original function (except for
injectivity, which is not needed here).}) which is
of Baire class 1 (hence it is also in $\sf{D}_\omega$) but has the property
that for \emph{any}
countable partition
$\langle A_n\mid n \in \omega \rangle$ of $\RR$ there is some $n_0
\in \omega$ such that $\restr{P}{A_{n_0}}$ is not continuous 
(see Lemma 5.4 in \cite{cichonmorayne}): this
means that $\tilde{\sf{D}}^\W_\xi \subsetneq \sf{D}_\xi$ for every
$\xi \geq \omega$. 
Since both $\sf{D}^\W_\xi$ and $\tilde{\sf{D}}^\W_\xi$ are
Borel-amenable, all these observations show that
we have at least two (respectively,
three) Borel-amenable sets of reductions of level $\xi$ for $\xi>2$
(respectively, $\xi\geq \omega$).
Nevertheless note that the
counter-example $P$ does not allow to prove the failure of
Conjecture \ref{conj1} for finite levels since $P$ turns out to be a
``proper'' $\bDelta^0_\omega$-function, i.e.\  $P \notin
\sf{D}_n$ for any $n \in \omega$ (this fact will be explicitly
proved in the forthcoming \cite{mottorosbairereductions}): 
hence it remains an open problem to
determine if Conjecture \ref{conj1} holds when $2<\xi<\omega$.

One could think that Conjecture \ref{conj1} fails 
in the general case because is too strong,
and that the
Jayne-Rogers Theorem could admit a weaker generalization which
holds for every $\xi < \omega_1$. This objection suggests
to formulate the following Conjecture, in which we require that for any
$\bDelta^0_\xi$-function there must be some $\bDelta^0_\xi$-partition
$\langle A_n\mid n \in \omega \rangle$ of $\RR$ such that $\restr{f}{A_n}$ is
only ``simpler'' then $f$ (instead of continuous).

\begin{conj}\label{conj2}
Let $\xi <\omega_1$ be a nonzero ordinal. Then $f\colon \RR \to \RR$ is a
$\bDelta^0_\xi$-function if and only if there is a
$\bDelta^0_\xi$-partition $\langle A_n\mid n \in \omega \rangle$ of $\RR$ such
that, for every $n \in \omega$, $\restr{f}{A_n}$ is
$\bDelta^0_{\mu_n}$-function (for some $\mu_n< \xi$), i.e.\ $f^{-1}(D)$ is in $\bDelta^0_{\mu_n}$ \emph{relatively to $A_n$} whenever $D \in \bDelta^0_{\mu_n}$.
\end{conj}

One direction is trivial, but also this Conjecture
fails for $\xi = \omega$ if we assume $\DCR$. The proof of this fact heavily rely on a deep result of Solecki
which will be used to prove Proposition \ref{propstrongdec}.
Let $X_1$, $Y_1$, $X_2$ and $Y_2$ be
separable metric spaces and pick any $f\colon X_1 \to Y_1$ and $g\colon X_2 \to Y_2$: 
then we say that \emph{$f$ is contained
in $g$} ($f \sqsubseteq g$ in symbols) if and only if there are embeddings
$\fhi\colon X_1 \to X_2$ and $\psi\colon  f(X_1) \subseteq Y_1 \to Y_2$ such that $\psi
\circ f = g \circ \fhi$.
This notion of containment between functions
allows to bound the complexity of a function by showing that it is
contained in another function of known complexity. In fact it is
straightforward to check that if $f$ and $g$ are as above we have that if $f \sqsubseteq
g$ and $g$ is a $\bDelta^0_\xi$-function (for some nonzero $\xi <
\omega_1$) then also $f$ is a $\bDelta^0_\xi$-function
(similarly, if $g$ is of Baire class $\xi$ then also $f$ is of Baire class
$\xi$), and conversely if $f$ is \emph{not} a $\bDelta^0_\xi$-function (or a Baire 
class $\xi$ function) then also $g$ is \emph{not} a $\bDelta^0_\xi$-function
(or a Baire class $\xi$ function). 

\begin{proposition}[\ZFC]\label{propstrongdec}
  Let $X$ be a Polish space and $Y$ be a separable metric space. For any
Baire class $1$ function
$f\colon X \to Y$ either there is some countable partition $\seq{X_n}{n \in \omega}$
of $X$ such that $\restr{f}{X_n}$ is continuous for every $n \in \omega$, or else 
for every
Borel partition (equivalently, $\bSigma^1_1$-partition)
$\langle A_n\mid n \in \omega \rangle$ of $X$ there is some $k \in
\omega$ such that $P \sqsubseteq \restr{f}{A_k}$.
\end{proposition}

\begin{proof}
Assume that the first alternative does not hold.
  Since each $A_n$ is an analytic subset of a Polish space, by definition it is
also Souslin and hence we can apply Solecki's Theorem 4.1 of  \cite{solecki} to 
$\restr{f}{A_n}$. But by our assumption it can not be the case that the first alternative
of Solecki's Theorem holds for each $\restr{f}{A_n}$, thus  
the second alternative must hold for some index $k
\in \omega$, that is $P \sqsubseteq \restr{f}{A_k}$.
\end{proof}

In particular, for every Borel partition $\langle A_n\mid n \in \omega
\rangle$ of $\RR$ there is some $k \in \omega$ such that
$\restr{P}{A_k} \sqsubseteq P$ and $P \sqsubseteq \restr{P}{A_k}$, i.e.\
there is some piece of the partition on which $P$ has ``maximal
complexity''.
Proposition \ref{propstrongdec} is proved using $\AC$ (since Solecki's Theorem was) 
but, since its statement is projective, it
is true also in any model of $\ZF+\DCR$ by absoluteness (see Lemma 19 of
\cite{andrettamoreonwadge}). This means that, under $\ZF+\DCR$,
for every Borel partition $\langle A_n\mid n \in \omega \rangle$
of $\RR$ there is some $k \in \omega$ such that $\restr{P}{A_k}$
is $\bDelta^0_\omega$ but not $\bDelta^0_n$ (for any $n \in
\omega$), and thus Conjecture
\ref{conj2} fails for $\xi = \omega$.

Therefore we have proved that the Jayne-Rogers Theorem does not admit
any generalization that holds for \emph{every} level of the Borel hierarchy.
Nevertheless, we have also a positive result: in fact
Theorem \ref{theorequivalentreductions} implies that we can extend
its corollary mentioned above (assuming at least $\SLO^{ \sf{D}^\W_\xi}+\NFS$) 
for \emph{every} possible index $\xi \leq 
\omega_1$, that is we can prove that $\sf{D}_\xi \simeq \sf{D}^\W_\xi$ (clearly
this is nontrivial, as we have seen, for $\xi \geq 3$).
Thus, in particular, $\leq_{\sf{D}^\W_\xi}$ and
$\leq_{\sf{D}_\xi}$ give rise to the same structure for every
countable $\xi$. The same is true (under $\SLO^\F+\NFS$) if
we replace $\sf{D}^\W_\xi$ with any Borel-amenable set of
reductions  $\F$
 of level $\xi$, that is, by previous observations, for any $\F$
 such that $\sf{D}^\Lip_\xi \subseteq
\F \subseteq \sf{D}_\xi$ (and in this case it is easy to check that
we have also $\Sat(\F)=\sf{D}_\xi$).
To appreciate this result once again, note that for \emph{every}
$\xi \leq \omega_1$ (hence also for the simplest
cases
$\xi=1,2$) we have that $\sf{D}^\Lip_\xi \subsetneq \sf{D}^\W_\xi$ and thus
also $\sf{D}^\Lip_\xi \subsetneq \sf{D}_\xi$ (therefore
we get other
examples of Borel-amenable sets of reductions for each level). 
 In fact 
it is easy to check that
$f \colon \RR \to \RR\colon x \mapsto \langle x(2n+1)\mid n \in \omega \rangle$
is a \emph{uniformly} continuous function such that $\restr{f}{\bN_s}$ is not Lipschitz
for any $s \in {}^{<\omega}\omega$.
Since by the Baire Category Theorem for \emph{any} partition $\seq{A_n}{n \in \omega}$ of $\RR$ 
there must be
some $n_0 \in \omega$ and a sequence $s \in {}^{<\omega}\omega$ such that
$\bN_s \subseteq \Cl(A_{n_0})$, using an argument similar to the one of Claim \ref{claim6.2.1}
one can check that if $g \colon \RR \to \RR$ is such that $\restr{f}{A_{n_0}} = \restr{g}{A_{n_0}}$ then $g \notin \Lip$: thus, in particular, $f \notin \sf{D}^\Lip_\xi$ for any $\xi \leq \omega_1$.

\section{Some construction principles: nonself\-dual successor
degrees}\label{sectionsuccessor}

We have seen that if $\F \in \sf{BAR}$ then the structure of degrees
associated to it looks like the Wadge one. We now want to go
further and show how to construct, given a selfdual degree, its successor
degree(s) (the successor of a nonselfdual pair can easily be obtained
with the $\oplus$ operation, see Theorem
\ref{theorgeneralproperties}). Let $\xi \leq \omega_1$ be the level of $\F$. 
If $\xi =
\omega_1$ we can appeal to the fact that $\F
\simeq \Bor$ and that the case $\F=\Bor$ has been
already treated in \cite{andrettamartin}, hence we have only to consider the case
$\xi<\omega_1$.

\emph{From this point on we will assume $\SLOL+\NFS+\DCR$ for the rest of this Section}.
Following \cite{andrettamartin} again, at the beginning it is convenient to deal with
$\F$-pointclasses rather than $\F$-degrees (but we will show later in this Section how to avoid them and directly construct successor degrees). Notice that every
$\F$-pointclass is also a boldface pointclass, since by Theorem
\ref{theorequivalentreductions} we have $\F \simeq \sf{D}_\xi
\supseteq \sf{D}_1$. 
First note that from the analysis of the previous Sections the first
nontrivial  (i.e.\ different from $\{ \emptyset\}$ and $\{\RR\}$)
$\F$-pointclass is $\bDelta^0_\xi$ (which is selfdual), and it is followed by the
nonselfdual pointclasses $\bSigma^0_\xi$ and $\bPi^0_\xi$.
Moreover one can easily check that $\bGamma$ is a nonselfdual $\F$-pointclass
if and only if
$\bGamma = \{B \subseteq \RR\mid B \leq_\L A\}$ for some $A
  \subseteq \RR$ such that $A \nleq_\F \neg A$
if and only if 
$\bGamma$ is an $\F$-pointclass which admits a universal
  set.
So let
$\bGamma$ be any nonselfdual $\F$-pointclass and let $A \nleq_\F \neg
A$ be such that $\bGamma = \{B \subseteq \RR\mid B \leq_\L A \}$. Define
\[ \bGamma^* = \{ (F \cap X) \cup (F' \setminus X')\mid F,F' \in
\bSigma^0_\xi, F \cap F' = \emptyset \text{ and }X,X' \in \bGamma\}\]
and $\bDelta^* = \bGamma^* \cap (\bGamma^*)\breve{}$. Using the fact
that $\bSigma^0_\xi$ has the reduction property, we can argue as in
\cite{andrettamartin} to show that $\bGamma^*$ is a nonselfdual
$\F$-pointclass which contains both $\bGamma$ and $\breve{\bGamma}$
and is such that $\bDelta^* = \{ B \subseteq \RR\mid B \leq_\F A \oplus
\neg A\}$. Therefore $\{\bGamma^* \setminus
\bDelta^*,(\bGamma^*)\breve{}\setminus \bDelta^*\}$ is the first
nonselfdual pair above $[A \oplus \neg A]_\F$, i.e.\ it is formed by the
successor degrees of $[A \oplus \neg A]_\F =
\bDelta^* \setminus(\bGamma \cup \breve{\bGamma})$.
Similarly, if $\langle\bGamma_n\mid n \in \omega \rangle$ is a
strictly increasing sequence of nonselfdual $\F$-pointclasses and
$\bGamma_n = \{B \subseteq \RR\mid B \leq_\L A_n\}$, we can define
\[
\begin{split}
\bLambda = \Bigl\{ \bigcup\nolimits_{n\in\omega}(F_n \cap X_n)\mid \;&F_n \in \bSigma^0_\xi,
F_n \cap F_m = \emptyset \text{ if }n \neq m, \text{ and}\\
& X_n \in
\bGamma_n \text{ for every }n \in \omega\Bigr\}
\end{split} \]
and $\bDelta = \bLambda \cap \breve{\bLambda}$. Using the generalized
reduction property of $\bSigma^0_\xi$, one can prove again that
$\bLambda$ is a nonselfdual $\F$-pointclass wich contains each
$\bGamma_n$ and such that $\bDelta = \{ B \subseteq \RR\mid B \leq_\F
\bigoplus_n A_n\}$. Thus $\{ \bLambda \setminus \bDelta,
\breve{\bLambda} \setminus \bDelta\}$ is the first nonselfdual pair above
$[\bigoplus_n A_n]_\F$ and is formed by the successor degrees of
$[\bigoplus_n A_n]_\F = \bDelta
\setminus (\bigcup_n \bGamma_n)$.

This analysis allows us to give a complete description of the
first $\omega_1$ levels of the $\leq_\F$ hierarchy: in particular,
one can inductively show that the $\alpha$-th pair of nonselfdual
$\F$-pointclasses (for $\alpha < \omega_1$) is formed by
$\alpha\text{-}\bSigma^0_\xi$ and its dual (for the
definition of the difference pointclasses $\alpha$-$\bGamma$ see
\cite{kechris}).

Now put $\bPi^0_{<\xi} = \bigcup_{\mu<\xi} \bPi^0_\mu$ and let $A
\subseteq \RR$ be any set such that $A \leq_\F \neg A$: we
can ``summarize'' the constructions above by showing that
\[
\begin{split}\bGamma^+(A) =\Bigl\{\bigcup\nolimits_{n \in\omega}(F_n \cap A_n)\mid \;&F_n \in
  \bPi^0_{<\xi},
F_n \cap F_m = \emptyset \text{ for }n \neq m, \text{ and}\\
&A_n <_\F
A\text{ for every }n \in \omega\Bigr\}
\end{split}\]
and its dual are the smallest nonselfdual $\F$-pointclasses which
contain $A$. To see this, we must first consider two cases: if $A$ is a
successor with respect to $\leq_\F$ (i.e.\ $A \equiv_\F C \oplus \neg C$ for
some $C \nleq_\F \neg C$) then $\bGamma^* \subseteq
\bGamma^+(A)$ (where $\bGamma^*$ is obtained from $\bGamma =\{B \subseteq \RR\mid B \leq_\L C\}$
as before),
while if $A$ is limit then there is a strictly increasing sequence of
nonselfdual
$\F$-pointclasses $\bGamma_n = \{B \subseteq \RR\mid B \leq_\L A_n\}$
such that $A \equiv_\F \bigoplus_n A_n$, and it is not hard to see that
$\bLambda \subseteq \bGamma^+(A)$, where $\bLambda$ is constructed
from the $\bGamma_n$'s as
above. Since $\bGamma^+(A)$ is clearly an $\F$-pointclass, the result
will follow if we can prove that if $B
\subseteq \RR$ is such that $B,\neg B \in \bGamma^+(A)$ then $B
\leq_\F A$. So let $B = \bigcup_n(F_n \cap A_n)$ and $\neg B = \bigcup_n
(F'_n \cap A'_n)$. Since $\bigcup_n(F_n \cup F'_n) = \RR$ and
$\bSigma^0_\xi$ has the generalized reduction property, we can find
$\hat{F}_n,\hat{F}'_n \in \bDelta^0_\xi$ such that they form a
partition of $\RR$ and $\hat{F}_n \subseteq F_n,\;\hat{F}'_n \subseteq
F'_n$ for each $n$. Hence 
\begin{align*}
& {x \in \hat{F}_n} \imp (x \in B \iff x \in A_n)\\
& {x \in \hat{F}'_n} \imp (x \in B \iff x \notin A'_n),
\end{align*}
and since $A_n,\neg A'_n <_\F A$ by $\SLO^\F$, we get $B \leq_\F A$ by
Borel-amenability of $\F$ (see Proposition \ref{lemmapartition}).\\

Now we will show how to construct the
successor of an $\F$-selfdual degree $[A]_\F$ in a ``direct'' way,
i.e.\ without considering the associated $\F$-pointclasses. First fix
an increasing sequence of ordinals $\langle \mu_n\mid n \in \omega
\rangle$ cofinal in $\xi$ (clearly if $\xi = \nu+1$ we can take
$\mu_n=\nu$ for every $n \in \omega$), and a sequence of sets $P_n$
such that $P_n \in \bPi^0_{\mu_n} \setminus \bSigma^0_{\mu_n}$. 
Let $\langle \cdot,\cdot \rangle\colon \omega \times \omega \to \omega$ be
any  bijection,
e.g.\ $\langle n,m \rangle =2^n(2m+1)-1$. Then we can define the homeomorphism
\[ \bigotimes \colon  {}^\omega \RR \to \RR \colon  \langle x_n\mid  n \in
\omega \rangle \mapsto x=\bigotimes\nolimits_n x_n,\]
where $x(\langle n,m \rangle) = x_n(m)$, and, conversely, the ``projections''
$\pi_n\colon  \RR \to \RR$ defined by $\pi_n(x) = \langle
x(\langle n,m \rangle)\mid m \in \omega \rangle$ (clearly, every ``projection''
 is
surjective, continuous and open). Moreover, given  a sequence of functions $f_n\colon 
\RR \to \RR$, we can use the homeomorphism
$\bigotimes$ to  define
the function
\[ \bigotimes(\langle f_n\mid n\in\omega\rangle)=\bigotimes\nolimits_n f_n\colon 
\RR \to \RR \colon  x \mapsto \bigotimes\nolimits_n (f_n(x)).\]
It is not hard to check that $\bigotimes_n f_n$ is continuous if and
only  if all
the $f_n$'s are continuous.
Now
consider the set
\[ \Sigma^\xi(A) = \{x \in \RR\mid \exists n (\pi_{2n}(x) \in P_n
\wedge \forall i<n(\pi_{2i}(x) \notin P_i) \wedge \pi_{2n+1}(x) \in
A)\}.\]
We will prove that $\Sigma^\xi(A)$ is $\bGamma^+(A)$-complete, from
which it follows that $[\Sigma^\xi(A)]_\F$ is a (nonselfdual) successor of
$[A]_\F$. Inductively define $F_0 = \{x \in \RR\mid \pi_0(x) \in
P_0\}$ and $F_{n+1}=\{ x \in \RR\mid \pi_{2(n+1)}(x) \in P_{n+1}
\wedge \forall i \leq n (\pi_{2i}(x) \notin P_i)\}$. Clearly $F_n \in
\bDelta^0_\xi$, $F_n \cap F_m =\emptyset$ for $n \neq m$, and
$\Sigma^\xi(A) \subseteq \bigcup_n F_n$. Moreover put $A_n =
\pi^{-1}_{2n+1}(A)$, and let $\langle
D_{n,k}\mid k \in \omega \rangle$ be a $\bPi^0_{<\xi}$-partitions of $\RR$
such that $A_n \cap D_{n,k} <_\F A$ for every $k \in \omega$ (this partitions
must
exist by Theorem \ref{theorDP} if $A_n \equiv_\F A\equiv_\F
\neg A$: if instead $A_n <_\F A$ simply take 
$D_{n,0}=\RR$ and $D_{n,k+1} = \emptyset$). Finally, let $G_{n,m} \in \bPi^0_{<\xi}$ be such that
$G_{n,m} \cap G_{n,m'} =\emptyset$ if $m \neq m'$ and $F_n = \bigcup_m
G_{n,m}$. Thus
\[ \Sigma^\xi(A) = \bigcup\nolimits_{n,m,k \in \omega} ((G_{n,m} \cap D_{n,k}) \cap (D_{n,k} \cap A)),\]
and hence $\Sigma^\xi(A) \in \bGamma^+(A)$ by definition.

Conversely let $B = \bigcup_k(F_k \cap A_k)$ be a generic set in
$\bGamma^+(A)$ and let $n_k$ be an increasing sequence of natural
numbers such that $F_k \in \bPi^0_{\mu_{n_k}}$ (such a sequence must
exist since the sequence $\mu_n$ is cofinal in $\xi$). Fix $y_i
\notin P_i$ for every $i \in \omega$ and define
\begin{align*}
f_i &=
\begin{cases}
\text{a continuous reduction of }F_k \text{ to }P_{n_k} & \text{if }i=n_k\\
\text{the constant function with value }y_i & \text{otherwise}
\end{cases}\\
g_i & =
\begin{cases}
\text{a continuous reduction of }A_n \text{ to }A & \text{if }i =
n_k\\
\id & \text{otherwise}
\end{cases}
\end{align*}
($A_n$ is continuously reducible to $A$ by $\SLOL$ and the fact that
$A_n <_\F A$).
Finally put $h_{2i} = f_i$ and $h_{2i+1} = g_i$ for every $i \in
\omega$: it is easy to check that $f = \bigotimes_i h_i$ is continuous
and reduces $B$ to $\Sigma^\xi(A)$.

In a similar way one can prove that the set
\[ \Pi^\xi(A) = \Sigma^\xi(A) \cup P_\xi,\]
where $P_\xi = \{x\in \RR\mid \forall n(\pi_{2n}(x) \notin P_n) \}$,
is complete for the dual pointclass of $\bGamma^+(A)$, from which it
follows that $\Pi^\xi(A) \equiv_\F \neg \Sigma^\xi(A)$.
This construction suggests also  how to define certain games $G^\W_\xi$ which represent a full generalization 
for all the levels of the backtrack game (in the sense that the
legal strategies for player $\pII$ in $G^\W_\xi$ induce exactly the functions in
$\sf{D}^\W_\xi$). This seems to be useful
since it allows to use ``combinatorial'' arguments to prove results
about the $\sf{D}^\W_\xi$-hierarchies (and hence, by Proposition
\ref{proponlyDxi} and Theorem \ref{theorequivalentreductions},
also about the degree-structure induced by \emph{any} Borel-amenable
set of reductions).


\begin{thebibliography}{10}

\bibitem{andrettaslo}
Alessandro Andretta.
\newblock The \textsf{SLO} principle and the {W}adge hierarchy.
\newblock To appear.

\bibitem{andrettaequivalence}
Alessandro Andretta.
\newblock Equivalence between {W}adge and {L}ipschitz determinacy.
\newblock {\em Annals of Pure and Applied Logic}, 123:163--192, 2003.

\bibitem{andrettamoreonwadge}
Alessandro Andretta.
\newblock More on {W}adge determinacy.
\newblock {\em Annals of Pure and Applied Logic}, 144:2--32, 2006.

\bibitem{andrettamartin}
Alessandro Andretta and Donald~A. Martin.
\newblock Borel-{Wadge} degrees.
\newblock {\em Fundamenta Mathematicae}, 177(1):175--192, 2003.

\bibitem{cichonmorayne}
J.~Chico\'n{,} M. Morayne{,} J.~Pawlikowsky e~S.~Solecki.
\newblock Decomposing {B}aire functions.
\newblock {\em The Journal of Symbolic Logic}, 56(4):1273--1283, 1991.

\bibitem{jaynerogers}
J.~E. Jayne and C.~A. Rogers.
\newblock First level {B}orel functions and isomorphism.
\newblock {\em Journal de Mathematiques Pures et Appliques}, 61:177--205, 1982.

\bibitem{kechris}
Alexander~S. Kechris.
\newblock {\em Classical Descriptive Set Theory}.
\newblock Number 156 in Graduate Text in Mathematics. Springer-Verlag,
  Heidelberg, New York, 1995.

\bibitem{mottorosbairereductions}
Luca~Motto Ros.
\newblock Baire reductions and non-{B}orel-amenable reducibilities: toward a
  dichotomy for {B}orel reducibilities.
\newblock In preparation.

\bibitem{mottorossuperamenability}
Luca~Motto Ros.
\newblock Beyond {B}orel-amenability: scales and superamenable reductions.
\newblock In preparation.

\bibitem{mottorosthesis}
Luca~Motto Ros.
\newblock {\em General {R}educibilities for {S}ets of {R}eals}.
\newblock PhD thesis, Polytechnic of Turin, Italy, 2007.

\bibitem{mottorosnewcharacterization}
Luca~Motto Ros.
\newblock A new characterization of {B}aire class $1$ functions.
\newblock Internal report N. 1, Department of Mathematics, Polytechnic of
  Turin, january 2007.

\bibitem{solecki}
Slawomir Solecki.
\newblock Decomposing {B}orel sets and functions and the structure of {B}aire
  class 1 functions.
\newblock {\em Journal of the American Mathematical Society}, 11(3):521--550,
  1998.

\bibitem{steelthesis}
John~R. Steel.
\newblock {\em Determinateness and {S}ubsystems of {A}nalysis}.
\newblock PhD thesis, University of California, Berkeley, 1977.

\bibitem{wadgethesis}
William~W. Wadge.
\newblock {\em Reducibility and Determinateness on the Baire Space}.
\newblock PhD thesis, University of California, Berkeley, 1983.

\bibitem{vanwesepthesis}
Robert A.~Van Wesep.
\newblock {\em Subsystems of {S}econd-{O}rder {A}rithmetic and {D}escriptive
  {S}et {T}heory under the {A}xiom of {D}eterminateness}.
\newblock PhD thesis, University of California, Berkeley, 1977.

\bibitem{vanwesepwadgedegrees}
Robert A.~Van Wesep.
\newblock Wadge degrees and descriptive set theory.
\newblock In Alexander~S. Kechris and Yiannis~N. Moschovakis, editors, {\em
  Cabal Seminar 76-77}, number 689 in Lecture Notes in Mathematics.
  Springer-Verlag, 1978.

\end{thebibliography}
\end{document}